\documentclass[11pt]{article}
\usepackage[latin9]{inputenc}
\usepackage{color}
\usepackage{amsmath}
\usepackage{amsthm}
\usepackage{amssymb}
\PassOptionsToPackage{normalem}{ulem}
\usepackage{ulem}
\usepackage[unicode=true,pdfusetitle,
 bookmarks=true,bookmarksnumbered=false,bookmarksopen=false,
 breaklinks=false,pdfborder={0 0 1},backref=page,colorlinks=true]
 {hyperref}
\hypersetup{
 linkcolor=blue}

\makeatletter
\numberwithin{equation}{section}
\numberwithin{figure}{section}
\theoremstyle{plain}
\newtheorem{thm}{\protect\theoremname}[section]
\theoremstyle{plain}
\newtheorem{cor}[thm]{\protect\corollaryname}
\theoremstyle{remark}
\newtheorem{rem}[thm]{\protect\remarkname}
\theoremstyle{plain}
\newtheorem{prop}[thm]{\protect\propositionname}
\theoremstyle{plain}
\newtheorem{lem}[thm]{\protect\lemmaname}

\@ifundefined{date}{}{\date{}}
\usepackage{graphicx}
\usepackage{appendix}

\usepackage{enumitem}
\usepackage[backref=page]{hyperref}

\makeatother

\providecommand{\corollaryname}{Corollary}
\providecommand{\lemmaname}{Lemma}
\providecommand{\propositionname}{Proposition}
\providecommand{\remarkname}{Remark}
\providecommand{\theoremname}{Theorem}

\begin{document}
\global\long\def\F{\mathrm{\mathbf{F}} }%
\global\long\def\Aut{\mathrm{Aut}}%
\global\long\def\C{\mathbf{C}}%
\global\long\def\H{\mathbb{H}}%
\global\long\def\U{\mathbf{U}}%
\global\long\def\P{\mathcal{P}}%
\global\long\def\ext{\mathrm{ext}}%
\global\long\def\hull{\mathrm{hull}}%
\global\long\def\triv{\mathrm{triv}}%
\global\long\def\Hom{\mathrm{Hom}}%

\global\long\def\trace{\mathrm{tr}}%
\global\long\def\End{\mathrm{End}}%

\global\long\def\L{\mathcal{L}}%
\global\long\def\W{\mathcal{W}}%
\global\long\def\E{\mathbb{E}}%
\global\long\def\SL{\mathrm{SL}}%
\global\long\def\R{\mathbf{R}}%
\global\long\def\Pairs{\mathrm{PowerPairs}}%
\global\long\def\Z{\mathbf{Z}}%
\global\long\def\rs{\to}%
\global\long\def\A{\mathcal{A}}%
\global\long\def\a{\mathbf{a}}%
\global\long\def\rsa{\rightsquigarrow}%
\global\long\def\D{\mathbf{D}}%
\global\long\def\b{\mathbf{b}}%
\global\long\def\df{\mathrm{def}}%
\global\long\def\eqdf{\stackrel{\df}{=}}%
\global\long\def\ZZ{\mathcal{Z}}%
\global\long\def\Tr{\mathrm{Tr}}%
\global\long\def\N{\mathbf{N}}%
\global\long\def\std{\mathrm{std}}%
\global\long\def\HS{\mathrm{H.S.}}%
\global\long\def\e{\mathbf{e}}%
\global\long\def\c{\mathbf{c}}%
\global\long\def\d{\mathbf{d}}%
\global\long\def\AA{\mathbf{A}}%
\global\long\def\BB{\mathbf{B}}%
\global\long\def\u{\mathbf{u}}%
\global\long\def\v{\mathbf{v}}%
\global\long\def\spec{\mathrm{spec}}%
\global\long\def\Ind{\mathrm{Ind}}%
\global\long\def\half{\frac{1}{2}}%
\global\long\def\Re{\mathrm{Re}}%
\global\long\def\Im{\mathrm{Im}}%
\global\long\def\Rect{\mathrm{Rect}}%
\global\long\def\Crit{\mathrm{Crit}}%
\global\long\def\Stab{\mathrm{Stab}}%
\global\long\def\SL{\mathrm{SL}}%
\global\long\def\TF{\mathsf{TF}}%
\global\long\def\p{\mathfrak{p}}%
\global\long\def\j{\mathbf{j}}%
\global\long\def\uB{\underline{B}}%
\global\long\def\tr{\mathrm{tr}}%
\global\long\def\rank{\mathrm{rank}}%
\global\long\def\K{\mathcal{K}}%
\global\long\def\hh{\mathbb{H}}%
\global\long\def\h{\mathfrak{h}}%

\global\long\def\EE{\mathcal{E}}%
\global\long\def\PSL{\mathrm{PSL}}%
\global\long\def\G{\mathcal{G}}%
\global\long\def\Int{\mathrm{Int}}%
\global\long\def\acc{\mathrm{acc}}%
\global\long\def\awl{\mathsf{awl}}%
\global\long\def\even{\mathrm{even}}%
\global\long\def\z{\mathbf{z}}%
\global\long\def\id{\mathrm{id}}%
\global\long\def\CC{\mathcal{C}}%
\global\long\def\cusp{\mathrm{cusp}}%
\global\long\def\i{\mathrm{int}}%
\global\long\def\new{\mathrm{new}}%

\global\long\def\LL{\mathbb{L}}%
\global\long\def\M{\mathbb{M}}%

\title{Near optimal spectral gaps for hyperbolic surfaces}
\author{Will Hide and Michael Magee}

\maketitle
\date{\vspace{-5ex}}
\begin{abstract}
{\footnotesize{}We prove that if $X$ is a finite area non-compact
hyperbolic surface, then for any $\epsilon>0$, with probability tending
to one as $n\to\infty$, a uniformly random degree $n$ Riemannian
cover of $X$ has no eigenvalues of the Laplacian in $[0,\frac{1}{4}-\epsilon)$
other than those of $X$, and with the same multiplicities. }{\footnotesize\par}

{\footnotesize{}As a result, using a compactification procedure due
to Buser, Burger, and Dodziuk, we settle in the affirmative the question
of whether there exists a sequence of closed hyperbolic surfaces with
genera tending to infinity and first non-zero eigenvalue of the Laplacian
tending to $\frac{1}{4}$. }{\footnotesize\par}
\end{abstract}
{\small{}\tableofcontents{}}{\small\par}

\section{Introduction}

Let $X$ be a finite-area hyperbolic surface, that is, a smooth surface
with Riemannian metric of constant curvature $-1$. In this paper,
all hyperbolic surfaces are assumed complete and orientable. The Laplacian
operator $\Delta_{X}$ on $L^{2}(X)$ has spectrum in $[0,\infty)$.
The bottom of the spectrum is always a discrete eigenvalue at $0$
that is simple if and only if $X$ is connected.

If $X$ is non-compact, which is the first focus of our paper, then
the spectrum of $\Delta_{X}$ in $[0,\frac{1}{4})$ consists of finitely
many discrete eigenvalues and the spectrum is absolutely continuous
in $[\frac{1}{4},\infty)$ \cite{LP}. As such, the \emph{spectral
gap} between $0$ and the rest of the spectrum has size at most $\frac{1}{4}$. 

We are interested in the size of this spectral gap for random surfaces.
The random model we use is the random covering model from \cite{MN1,MageeNaudPuder,MageeNaud2}.
For any $n\in\N$, the space of degree $n$ Riemannian covering spaces
of $X$ is a finite set that we equip with the uniform probability
measure. We say that a family of events, depending on $n$, holds
\emph{asymptotically almost surely }(a.a.s.) if they hold with probability
tending to one as $n\to\infty$. We also note that any eigenvalue
of the Laplacian on $X$ will be an eigenvalue of any cover of $X$,
with at least as large multiplicity. Therefore these `old eigenvalues'
must always be taken into account. As we explain in $\S\S$\ref{subsec:Construction},
a theorem of Dixon \cite{Dixon} implies that a uniformly random degree
$n$ covering space of a fixed connected non-compact $X$ will be
connected a.a.s.

The first theorem of the paper is the following.
\begin{thm}
\label{thm:main-theorem}Let $X$ be a finite-area non-compact hyperbolic
surface. Let $X_{n}$ denote a uniformly random degree $n$ covering
space of $X$. For any $\epsilon>0$, a.a.s.
\[
\spec(\Delta_{X_{n}})\cap\left[0,\frac{1}{4}-\epsilon\right)=\spec(\Delta_{X})\cap\left[0,\frac{1}{4}-\epsilon\right)
\]
and the multiplicities on both sides are the same.
\end{thm}

This theorem is the analog, for finite-area non-compact hyperbolic
surfaces of Friedman's theorem \cite{Friedman} (formerly Alon's conjecture
\cite{Alon}) stating that random $d$-regular graphs have almost
optimal spectral gaps. Friedman also proposed in \cite{FriedmanRelative}
that a variant of Alon's conjecture should hold for random covers
of any fixed finite graph, and this extended conjecture was recently
proved by Bordenave and Collins \cite{BordenaveCollins}. 

If $X$ has Euler characteristic $\chi(X)=-1$, then a result of Otal
and Rosas \cite[Thm. 2]{OtalRosas} states that $\spec(\Delta_{X})\cap\left[0,\frac{1}{4}\right)$
is just the multiplicity one eigenvalue at $0$. Taking such an $X$
as the base surface in Theorem \ref{thm:main-theorem} we therefore
obtain the following corollary.
\begin{cor}
\label{cor:There-exist-finite-area}There exist connected finite-area
non-compact hyperbolic surfaces $X_{i}$ with $\chi(X_{i})\to-\infty$
and 
\[
\inf\left(\spec(\Delta_{X_{i}})\cap(0,\infty)\right)\to\frac{1}{4}.
\]
In fact one can take $\chi(X_{i})=-i$.
\end{cor}

The corresponding question for closed surfaces $X$ is as follows.
If $X$ is closed then the spectrum of $\Delta_{X}$ consists of eigenvalues
\[
0=\lambda_{0}(X)\leq\lambda_{1}(X)\leq\cdots\leq\lambda_{i}(X)\leq\cdots
\]
with $\lambda_{i}(X)\to\infty$ as $i\to\infty$. Suppose that $X_{i}$
are a sequence of closed hyperbolic surfaces with genera $g(X_{i})\to\infty$.
Huber proved in \cite{Huber} that in this case,
\[
\lim\sup\lambda_{1}(X_{i})\leq\frac{1}{4}.
\]
Therefore $\frac{1}{4}$ is an asymptotically optimal lower bound
for $\lambda_{1}(X_{i})$. In the same scenario, it was conjectured
in \cite{Buser} that $\lim_{i\to\infty}\lambda_{1}(X_{i})\to0$.
This was corrected in \cite{Buser2} where it was put forward that
`probably' there exist a sequence of $X_{i}$ as above with 
\[
\lim_{i\to\infty}\lambda_{1}(X_{i})=\frac{1}{4}.
\]
See \cite[Conj. 5]{wu2018small}, \cite[Problem 10.3]{Wright} for
more recent iterations of the question of whether $\lambda_{1}(X_{i})$
can go to $\frac{1}{4}$ as the genus tends to infinity. 

We are able to resolve this question here by combining Corollary \ref{cor:There-exist-finite-area}
with the `handle lemma' of Buser, Burger, and Dodziuk \cite{BBD};
see Brooks and Makover \cite[Lemma 1.1]{BrooksMakover1} for the extraction
of the lemma. We obtain
\begin{cor}
\label{cor:there-exist-compact}~There exist closed hyperbolic surfaces
$X_{i}$ with genera\linebreak{}
$g(X_{i})\to\infty$ and $\lambda_{1}(X_{i})\to\frac{1}{4}$. In fact
one can take $g(X_{i})=i$.
\end{cor}

\subsection{Prior work}

The analog of Theorem \ref{thm:main-theorem} was proved for conformally
compact infinite area hyperbolic surfaces by the second named author
and Naud in \cite{MageeNaud2}, following a previous work with an
intermediate result \cite{MN1}. The reader should see however \cite{MageeNaud2}
for details because the spectral theory is more subtle for infinite
area surfaces and moreover the results in \cite{MageeNaud2} go beyond
statements about $L^{2}$ eigenvalues.

The very first uniform spectral gap for certain combinatorial models
of random surfaces (the `Brooks---Makover' models), yielding both
finite-area non-compact and closed surfaces, appears in work of Brooks
and Makover \cite{BrooksMakover}. This spectral gap is non-explicit.

The analog of Theorem \ref{thm:main-theorem} for closed hyperbolic
surfaces is still unknown. The best known result in the random cover
model is due to the second named author, Naud, and Puder \cite{MageeNaudPuder},
building on \cite{MPasympcover}, that gives an a.a.s. relative\footnote{Here relative refers to disregarding eigenvalues of the base surface.}
spectral gap of size $\frac{3}{16}-\epsilon$ for random covers of
a closed hyperbolic surface. The fact that a Weil---Petersson random
closed hyperbolic surface enjoys spectral gap of size $\frac{3}{16}-\epsilon$
with probability tending to one as genus $\to\infty$ was proved independently
by Wu and Xue \cite{wu2021random} and Lipnowski and Wright \cite{lipnowski2021optimal}.
The first uniform spectral gap of size $\approx0.0024$ for Weil---Petersson
random closed surfaces was proved by Mirzakhani in \cite{MirzakhaniRandom}.
Prior to the current work, the first named author proved in \cite{HideWP}
that Weil---Petersson random surfaces of genus $g$ with $O(g^{\alpha})$
cusps, with $\alpha<\frac{1}{2}$, have an explicit positive uniform
spectral gap depending on $\alpha$ as $g\to\infty$, which coincides
(for any fixed $\epsilon>0$) with $\frac{3}{16}-\epsilon$ if $\alpha=0$.

The fact that all models of random closed hyperbolic surfaces are
currently stuck at $\lambda_{1}=\frac{3}{16}-\epsilon$ is an obstacle
to proving Corollary \ref{cor:there-exist-compact} by the more direct
method of proving that random closed surfaces of large genus have
almost optimal spectral gaps. This $\frac{3}{16}$ barrier bears explanation;
it also appears in the famous $\frac{3}{16}$ Theorem of Selberg \cite{SelbergFourier}.
Very roughly speaking, this barrier corresponds to fine control of
probability events concerning \emph{simple} closed geodesics in random
surfaces and not having as fine control on non-simple geodesics. This
splitting of geodesics into simple and non-simple goes back to work
of Broder and Shamir \cite{BroderShamir} in the graph theoretic setting.
The appearance of $\frac{3}{16}$ in \cite{SelbergFourier} is for
different reasons relating to Selberg's use of Weil's bounds \cite{Weil}
on Kloosterman sums.

The previous best results towards Corollary \ref{cor:There-exist-finite-area},
just for the case of $\chi(X_{i})\to-\infty$, all came from arithmetic
hyperbolic surfaces. The records were held by Selberg \cite{SelbergFourier}
(spectral gaps of size $\frac{3}{16}=0.1875$), Gelbart and Jacquet
\cite{GelbartJacquet} (existence of spectral gaps larger than $\frac{3}{16}$),
Luo, Rudnick, and Sarnak \cite{LRS} (spectral gaps of size $\frac{171}{784}\approx0.218$),
Kim and Shahidi \cite{KimShahidi,KimShahidi2} (spectral gaps of size
$\frac{77}{324}\approx0.23765$), and currently, Kim and Sarnak \cite[Appendix 2]{KIM}
(spectral gaps of size $\frac{975}{4096}\approx0.23804$).

The history of Corollary \ref{cor:there-exist-compact} began with
McKean's paper \cite{McKean,McKeanCorrection} where it was wrongly
asserted that the first eigenvalue of any compact hyperbolic surface
is at least $\frac{1}{4}$; this was disproved by Randol in \cite{Randol}
who actually showed that there can be arbitrarily many eigenvalues
in $(0,\frac{1}{4})$. Buser proved in \cite{Buser2} using Selberg's
$\frac{3}{16}$ theorem and the Jacquet---Langlands machinery \cite{JacquetLanglands}
that there exist closed hyperbolic surfaces $X_{i}$ with genera $g(X_{i})\to\infty$
and $\lambda_{1}(X_{i})\geq\frac{3}{16}$. Buser, Burger, and Dodziuk
proved in \cite{BBD} that one can get the slightly weaker result
$\lambda_{1}(X_{i})\to\frac{3}{16}$ but without using Jacquet---Langlands.
This turns out to be instrumental in the current work; see Lemma \ref{lem:handle}
below. 

Following these works, all progress towards Corollary \ref{cor:there-exist-compact}
ran parallel to the previously discussed progress to Corollary \ref{cor:There-exist-finite-area}.
In particular, the previous best known result towards Corollary \ref{cor:there-exist-compact}
comes from combining the Kim-Sarnak bound with either Jacquet---Langlands
or \cite{BBD} to obtain the first part of Corollary \ref{cor:there-exist-compact}
with $\frac{1}{4}$ replaced by $\frac{975}{4096}$.

\subsection{Overview of proof}

The proof of Theorem \ref{thm:main-theorem} follows the following
strategy. We view 
\[
X=\Gamma\backslash\H
\]
where $\Gamma$ is a discrete torsion free subgroup of $\PSL_{2}(\R)$
and $\H$ is the hyperbolic upper half plane. We explain in $\S$\ref{sec:Random-covers}
that the random degree $n$ covers of $X$ are parameterized by 
\[
\phi\in\Hom(\Gamma,S_{n})\mapsto X_{\phi}
\]
where $\phi$ is a uniformly random homomorphism of the free group
$\Gamma$ into the symmetric group $S_{n}$. Because we disregard
eigenvalues and eigenfunctions lifted from $X$, we restrict our attention
to the space $L_{\new}^{2}(X_{\phi})$ of functions that are orthogonal
to all lifts of $L^{2}$ functions from $X$. The strategy is then,
for any $s_{0}>\frac{1}{2}$, to asymptotically almost surely produce
a bounded resolvent operator
\[
R_{X_{\phi}}(s):L_{\new}^{2}(X_{\phi})\to H_{\new}^{2}(X_{\phi})
\]
where $H_{\new}^{2}(X_{\phi})\subset L_{\new}^{2}(X_{\phi})$ is a
suitable Sobolev space (see $\S$\ref{subsec:Function-spaces}) such
that for any $s\in[s_{0},1]$ the identity
\[
\left(\Delta_{X_{\phi}}-s(1-s)\right)R_{X_{\phi}}(s)=\mathrm{\mathrm{Id}_{L_{\new}^{2}(X_{\phi})}},
\]
makes sense and holds true; this will forbid new eigenvalues in $[0,s_{0}(1-s_{0})]$.

The way we build our resolvent is by a parametrix construction. The
fundamental structure of parametrix construction is to produce a bounded
operator
\[
\M_{\phi}(s):L_{\new}^{2}(X_{\phi})\to H_{\new}^{2}(X_{\phi})
\]
 such that 
\[
\left(\Delta_{X_{\phi}}-s(1-s)\right)\M_{\phi}(s)=\mathrm{\mathrm{Id}_{L_{\new}^{2}(X_{\phi})}}+\LL_{\phi}(s)
\]
where we aim to prove $\LL_{\phi}(s)$ has norm less than one as a
bounded operator on $L_{\new}^{2}(X_{\phi})$ so that we can obtain
our bounded resolvent
\[
R_{X_{\phi}}(s)=\M_{\phi}(s)\left(\mathrm{\mathrm{Id}_{L_{\new}^{2}(X_{\phi})}}+\LL_{\phi}(s)\right)^{-1}.
\]

Often in a parametrix construction one also wants the $\LL$ term
to be compact: this will also be the case here and turns out to be
essential for the application of random operator results.

The way we build $\M_{\phi}(s)$ is by patching together a `cuspidal
parametrix' $\M_{\phi}^{\cusp}(s)$ based on a model resolvent in
the cusps and an an interior parametrix $\M_{\phi}^{\i}(s)$ that
only depends on the localization of its argument to a compact part
of $X_{\phi}$. We then let 
\[
\M_{\phi}(s)=\M_{\phi}^{\i}(s)+\M_{\phi}^{\cusp}(s)
\]
and we get a resulting splitting
\[
\LL_{\phi}(s)=\LL_{\phi}^{\i}(s)+\LL_{\phi}^{\cusp}(s).
\]
In $\S$\ref{sec:Cusp-parametrix} we show that for any $\phi$, the
term $\M_{\phi}^{\cusp}(s)$ can be designed so that $\|\LL_{\phi}^{\cusp}(s)\|\leq\frac{1}{5}$
(or any small number), so that it will not essentially interfere with
our plans of obtaining $\|\LL_{\phi}(s)\|<1$. The ability to do this
is a particular feature of the geometry of hyperbolic cusps.

The term $\M_{\phi}^{\i}(s)$ is based on averaging the resolvent
kernel of the hyperbolic plane over the fundamental group of $\Gamma$
(suitably twisting by $\phi$) to obtain an integral operator on $L_{\new}^{2}(X_{\phi})$.
The problem with this is that the averaging will not obviously converge,
so we have to multiply the hyperbolic resolvent kernel by a radial
cutoff that localizes to radii $\leq T+1$ to get a priori convergence
for all $s\in(\frac{1}{2},1]$. This gives us that $\M_{\phi}^{\i}(s)$
is bounded (Lemma \ref{lem:parametrix-bounded}).

The effect of this cutoff is that the error term $\LL_{\phi}^{\i}(s)$
is an integral operator with smooth kernel. We prove that we can unitarily
conjugate $\LL_{\phi}^{\i}(s)$ to
\[
\sum_{\gamma\in\Gamma}a_{\gamma}(s)\otimes\phi(\gamma)
\]
acting on $L^{2}(F)\otimes V_{n}^{0}$, where $F$ is a Dirichlet
fundamental domain for $\Gamma$ and $V_{n}^{0}$ is the standard
$n-1$ dimensional irreducible representation of $S_{n}$. The $a_{\gamma}(s)$
are compact operators on $L^{2}(F)$ and there are only finitely many
$\gamma\in\Gamma$ for which $a_{\gamma}(s)$ is non-zero. 

If instead, the $a_{\gamma}(s)$ were elements of $\End(\C^{r})$
for some fixed finite $r$, because $\Gamma$ is free we would now
be exactly in the situation to apply the breakthough results of Bordenave
and Collins from \cite{BordenaveCollins}. These random operator results,
combined with a linearization trick of Pisier \cite{Pisierlinearization},
would tell us that for any $\epsilon>0$, a.a.s. 
\begin{equation}
\|\LL_{\phi}^{\i}(s)\|\leq\left\Vert \sum_{\gamma\in\Gamma}a_{\gamma}(s)\otimes\rho_{\infty}(\gamma)\right\Vert +\epsilon\label{eq:infinity-operator}
\end{equation}
where $\rho_{\infty}:\Gamma\to\End(\ell^{2}(\Gamma))$ is the right
regular representation. Because the $a_{\gamma}(s)$ are in reality
compact operators on Hilbert spaces we can approximate by finite rank
operators to the same effect.

The key point (which is a sticking point in other approaches to this
problem using e.g. transfer operators) is that we understand the operator
in the right hand side of (\ref{eq:infinity-operator}) well: it can
be unitarily conjugated to an operator on $L^{2}(\H)$ that is the
composition of multiplication with a cutoff (with norm $\leq1$) and
an integral operator with real-valued radial kernel that localizes
to radii in $[T,T+1]$.

The key is that as this latter operator is self-adjoint we can use
the theory of the Selberg transform to estimate its norm in Lemma
\ref{lem:LH-noorm-bound}. By choosing $T$ sufficiently large in
the beginning, we can force the norm in the right hand side of (\ref{eq:infinity-operator})
to be as small as we like, given any $s_{0}>\frac{1}{2}$, for all
$s\geq s_{0}$.

Assembling these arguments, for any $s\geq s_{0}>\frac{1}{2}$, a.a.s.
$1+\LL_{\phi}(s)$ can be inverted which rules out $X_{\phi}$ having
a new eigenvalue at $s(1-s)$. To be able to get this result for all
$s\geq s_{0}$ with probability tending to one, we make sure that
$\|\LL_{\phi}(s)\|\leq\frac{3}{5}$ at a fine enough net of $s\in[s_{0},1]$
and then use a deviations estimate (Lemma \ref{lem:deviations}) to
show (deterministically) that $\|\LL_{\phi}(s)\|$ fluctuates at most
by $\frac{1}{5}$ from point to point.

\emph{A historical remark:} a similar parametrix method in the context
of hyperbolic surfaces (albeit for completely different purposes)
goes back to work of Guillopé and Zworski \cite{GZ1} who used the
method to give sharp upper bounds on the number of resonances of geometrically
finite hyperbolic surfaces in balls. In turn this method is based
on Vodev's `impressive refinement of the Fredholm determinant method'
\emph{(ibid.)} \cite{Vodev1,Vodev2,Vodev3} for the control of scattering
poles of perturbed Laplacians in Euclidean spaces.

\subsection*{Acknowledgments}

We thank A. Gamburd, C. Bordenave, B. Collins, F. Naud, D. Puder,
and P. Sarnak for conversations about this work. This project has
received funding from the European Research Council (ERC) under the
European Union\textquoteright s Horizon 2020 research and innovation
programme (grant agreement No 949143).

\section{Random covers\label{sec:Random-covers}}

\subsection{Construction\label{subsec:Construction}}

Let $X$ be a finite area non-compact connected hyperbolic surface.
We view $X$ as 
\[
X=\Gamma\backslash\H
\]
 where 
\[
\H=\{\,x+iy\,:\,x,y\in\R,\,y>0\,\}
\]
with metric
\[
\frac{dx^{2}+dy^{2}}{y^{2}}
\]
 where $\Gamma$ is a discrete torsion-free subgroup of $\PSL_{2}(\R)$
that acts, via Möbius transformations, by orientation preserving isometries
on $\H$.

For $n\in\N$ let $[n]\eqdf\{1,\ldots,n\}$ and $S_{n}$ denote the
group of permutations of $[n]$. For any homomorphism $\phi\in\Hom(\Gamma,S_{n})$
we construct a hyperbolic surface as follows. Let $\Gamma$ act on
$\H\times[n]$ by 
\[
\gamma(z,x)\eqdf(\gamma z,\phi(\gamma)[x])
\]
and let 
\[
X_{\phi}\eqdf\Gamma\backslash_{\phi}\left(\H\times[n]\right)
\]
denote the quotient by this action. If we choose $\phi$ uniformly
at random in $\Hom(\Gamma,S_{n})$, the resulting $X_{\phi}$ is a
uniformly random degree $n$ covering space of $X$\@. Note that
$\Gamma$ is a free group freely generated by some
\[
\gamma_{1},\ldots,\gamma_{d}\in\Gamma
\]
 and choosing $\phi$ is the same as choosing 
\[
\sigma_{i}\eqdf\phi(\gamma_{i}),\quad i=1,\ldots,d
\]
independently and uniformly at random in $S_{n}$.

The surface $X_{\phi}$ is connected if and only if $\Gamma$ acts
transitively on $[n]$ via $\phi$. By a theorem of Dixon \cite{Dixon},
two independent and uniformly random permutations in $S_{n}$ generate
$S_{n}$ or $A_{n}$ a.a.s and it follows that a uniformly random
cover $X_{\phi}$ is connected a.a.s.

Let $V_{n}\eqdf\ell^{2}([n])$ and $V_{n}^{0}\subset V_{n}$ the subspace
of functions on $[n]$ with zero mean. The representation of $S_{n}$
on $\ell^{2}([n])$ is its standard representation by 0-1 matrices
and the subspace $V_{n}^{0}$ is an irreducible subspace of dimension
$(n-1)$: we write 
\[
\rho_{\phi}:\Gamma\to\End(V_{n}^{0})
\]
 for the random representation of $\Gamma$ induced by the random
$\phi$. 

\subsection{Function spaces\label{subsec:Function-spaces}}

We define 
\[
L_{\new}^{2}(X_{\phi})
\]
 to be the space of $L^{2}$ functions on $X_{\phi}$ that are orthogonal
to all pullbacks of $L^{2}$ functions from $X$. The elements $f\in L_{\new}^{2}(X_{\phi})$
have mean value 0 fiber-wise in the sense that for almost every $x\in X$
we have $\sum_{i=1}^{n}f(x_{i})=0$ where $x_{i}$ are the lifts of
$x$ to $X_{\phi}$.

We have 
\[
L^{2}(X_{\phi})\cong L_{\new}^{2}(X_{\phi})\oplus L^{2}(X).
\]
This induces a multiplicity respecting inclusion $\spec(\Delta_{X})\subset\spec(\Delta_{X_{\phi}})$.
All other eigenvalues of $\Delta_{X_{\phi}}$(with multiplicities)
arise from eigenfunctions in the subspace $L_{\new}^{2}(X_{\phi})$.

Let $F$ denote a Dirichlet fundamental domain for $X$, that is,
for some fixed $o\in\H$, 
\[
F\eqdf\bigcap_{\gamma\in\Gamma\backslash\{\mathrm{id}\}}\{\,z\in\H\,:\,d(o,z)<d(z,\gamma o)\,\}.
\]
This choice will be convenient for the proof of Lemma \ref{lem:geometric}.
Let $C^{\infty}(\H;V_{n}^{0})$ denote the smooth $V_{n}^{0}$-valued
functions on $\H$. There is an isometric linear isomorphism between
\[
C^{\infty}(X_{\phi})\cap L_{\new}^{2}(X_{\phi})
\]
and the space of smooth $V_{n}^{0}$-valued functions on $\H$ satisfying
\begin{equation}
f(\gamma z)=\rho_{\phi}(\gamma)f(z)\label{eq:automorphy}
\end{equation}
 for all $\gamma\in\Gamma$, with finite norm
\[
\|f\|_{L^{2}(F)}^{2}\eqdf\int_{F}\|f(z)\|_{V_{n}^{0}}^{2}d\H<\infty.
\]
We will denote $C_{\phi}^{\infty}(\H;V_{n}^{0})\subset C^{\infty}(\H;V_{n}^{0})$
for these functions. Under this isomorphism, the Laplacian on $C^{\infty}(X_{\phi})\cap L_{\new}^{2}(X_{\phi})$
is intertwined with the Laplacian that acts on $C_{\phi}^{\infty}(\H;V_{n}^{0})$
in the obvious way (this can be defined by choosing any basis of $V_{n}^{0}$
and letting the Laplacian act coordinatewise on $V_{n}^{0}$-valued
functions). The completion of $C_{\phi}^{\infty}(\H;V_{n}^{0})$ with
respect to $\|\bullet\|_{L^{2}(F)}$ is denoted by $L_{\phi}^{2}(\H;V_{n}^{0})$;
the isomorphism above extends to one between $L_{\new}^{2}(X_{\phi})$
and $L_{\phi}^{2}(\H;V_{n}^{0})$. 

Let $C_{c,\phi}^{\infty}(\H;V_{n}^{0})$ denote the elements of $C_{\phi}^{\infty}(\H;V_{n}^{0})$
that are compactly supported when restricted to $\bar{F}$ (in other
words, compactly supported modulo $\Gamma$).

We also consider the following Sobolev spaces. Let $H^{2}(\H)$ denote
the completion of $C_{c}^{\infty}(\H)$ with respect to the norm
\[
\|f\|_{H^{2}(\H)}^{2}\eqdf\|f\|_{L^{2}(\H)}^{2}+\|\Delta f\|_{L^{2}(\H)}^{2}.
\]
We let $H_{\phi}^{2}(\H;V_{n}^{0})$ denote the completion of $C_{c,\phi}^{\infty}(\H;V_{n}^{0})$
with respect to the norm
\[
\|f\|_{H_{\phi}^{2}(\H;V_{n}^{0})}^{2}\eqdf\|f\|_{L^{2}(F)}^{2}+\|\Delta f\|_{L^{2}(F)}^{2}.
\]
We let $H^{2}(X_{\phi})$ denote the completion of $C_{c}^{\infty}(X_{\phi})$
with respect to the norm 
\[
\|f\|_{H^{2}(X_{\phi})}^{2}\eqdf\|f\|_{L^{2}(X_{\phi})}^{2}+\|\Delta f\|_{L^{2}(X_{\phi})}^{2}.
\]
Viewing $H^{2}(X_{\phi})$ as a subspace of $L^{2}(X_{\phi})$ in
the obvious way, we let 
\[
H_{\new}^{2}(X_{\phi})\eqdf H^{2}(X_{\phi})\cap L_{\new}^{2}(X_{\phi}).
\]
Similarly to before there is an isometric isomorphism between $H_{\new}^{2}(X_{\phi})$
and $H_{\phi}^{2}(\H;V_{n}^{0})$ that intertwines the two relevant
Laplacian operators.

\section{Background random matrix theory}

In this section we give an account of the breakthrough result of Bordenave
and Collins \cite{BordenaveCollins} and its amplification through
a linearization trick that appears in Pisier \cite[Cor. 14]{Pisierlinearization}
(see also, historically, \cite{PisierOT,HaagerupThr}, and \cite[\S 6]{BordenaveCollins}).

Let $\std_{n}:S_{n}\to\End(V_{n}^{0})$ denote the linear action of
$S_{n}$ on $V_{n}^{0}$. We let $\rho_{\infty}:\Gamma\to\End(\ell^{2}(\Gamma))$
denote the right regular representation of $\Gamma$. We are going
to use the following direct consequence of \cite[Thm. 2]{BordenaveCollins}.
\begin{thm}[Bordenave---Collins]
\label{thm:Bordenave-Collins}Suppose that $r\in\N$ and $a_{0},a_{1},\ldots a_{d}\in\mathrm{Mat}_{r\times r}(\C)$.
Suppose $a_{0}^{*}=a_{0}$. Then for any $\epsilon>0$, with probability
tending to one as $n\to\infty$, we have 
\begin{align*}
 & \left\Vert a_{0}\otimes\mathrm{Id}_{V_{n}^{0}}+\sum_{i=1}^{d}a_{i}\otimes\std_{n}(\sigma_{i})+a_{i}^{*}\otimes\std_{n}(\sigma_{i}^{-1})\right\Vert _{\C^{r}\otimes V_{n}^{0}}\\
\leq & \left\Vert a_{0}\otimes\mathrm{Id}_{\ell^{2}(\Gamma)}+\sum_{i=1}^{d}a_{i}\otimes\rho_{\infty}(\gamma_{i})+a_{i}^{*}\otimes\rho_{\infty}(\gamma_{i}^{-1})\right\Vert _{\C^{r}\otimes\ell^{2}(\Gamma)}+\epsilon.
\end{align*}
The norm on the top is the operator norm on $\C^{r}\otimes V_{n}^{0}$.
The norm on the bottom is the operator norm on $\C^{r}\otimes\ell^{2}(\Gamma)$.
\end{thm}

\begin{rem}
Note in the above that $\std_{n}(\sigma_{i})=\rho_{\phi}(\gamma_{i})$.
\end{rem}

This will be combined with the following result of Pisier \cite[Cor. 14 and following remark]{Pisierlinearization}.
\begin{prop}[Pisier]
\label{prop:Pisier}The result of Theorem \ref{thm:Bordenave-Collins}
implies that for any $R,k\in\N$, any non-commutative polynomials
\[
P_{1},P_{2},\ldots,P_{k}
\]
in the variables $\sigma_{1},\ldots,\sigma_{d},\sigma_{1}^{-1},\cdots,\sigma_{d}^{-1}$
with complex coefficients, and any $a_{1},\ldots,a_{k}\in\mathrm{Mat}_{R\times R}(\C)$
we have the following. For any $\epsilon>0$, with probability tending
to one as $n\to\infty$, we have 
\begin{align}
 & \left\Vert \sum_{i=1}^{k}a_{i}\otimes\std_{n}(P_{i}(\sigma_{1},\ldots,\sigma_{d},\sigma_{1}^{-1},\ldots,\sigma_{d}^{-1}))\right\Vert _{\C^{R}\otimes V_{n}^{0}}\label{eq:operator-example}\\
\leq & \left\Vert \sum_{i=1}^{k}a_{i}\otimes\rho_{\infty}(P_{i}(\gamma_{1},\ldots,\gamma_{d},\gamma_{1}^{-1},\ldots,\gamma_{d}^{-1}))\right\Vert _{\C^{R}\otimes\ell^{2}(\Gamma)}+\epsilon.\nonumber 
\end{align}
\end{prop}

This type of `linearization' is obviously extremely powerful and has
played a role in many of the major breakthroughs in random operators
in recent years \cite{HaagerupThr,Collins2014,BordenaveCollins}.
Notice that one way to obtain an operator as in (\ref{eq:operator-example})
is as 
\[
\sum_{\gamma\in\Gamma}a_{\gamma}\otimes\rho_{\phi}(\gamma)
\]
where all $a_{\gamma}\in\mathrm{Mat}_{R\times R}(\C)$ and there are
only finitely many non-zero $a_{\gamma}$. Therefore we obtain combining
Theorem \ref{thm:Bordenave-Collins} and Proposition \ref{prop:Pisier}
the following:
\begin{cor}
\label{cor:bc}For any $r\in\N$ and any finitely supported map $\gamma\in\Gamma\mapsto a_{\gamma}\in\mathrm{Mat}_{r\times r}(\C)$,
for any $\epsilon>0$, with probability tending to one as $n\to\infty$
we have 
\[
\left\Vert \sum_{\gamma\in\Gamma}a_{\gamma}\otimes\rho_{\phi}(\gamma)\right\Vert _{\C^{r}\otimes V_{n}^{0}}\leq\left\Vert \sum_{\gamma\in\Gamma}a_{\gamma}\otimes\rho_{\infty}(\gamma)\right\Vert _{\C^{r}\otimes\ell^{2}(\Gamma)}+\epsilon.
\]
\end{cor}

No originality is claimed here: Corollary \ref{cor:bc} is essentially
already contained in \cite{BordenaveCollins} and in fact, \cite[Thm. 3]{BordenaveCollins}
is the version of this corollary when all $a_{i}$ are scalar.

\section{Cusp parametrix\label{sec:Cusp-parametrix}}

To simplify notation, we assume there is just one cusp of the finite-area
non compact hyperbolic surface $X$. This does not affect the arguments.
It follows from the Collar Lemma \cite[Thm. 4.4.6]{BuserBook} that
this single cusp can be identified with
\[
\CC\eqdf(1,\infty)\times S^{1}
\]
 with the metric
\begin{equation}
\frac{dr^{2}+dx^{2}}{r^{2}};\label{eq:cusp-metric}
\end{equation}
here the $x$ coordinate is in $S^{1}\eqdf\R/\Z$ and the $r$ coordinate
is in $[1,\infty)$. In the following, $\chi_{\CC}^{+},\chi_{\CC}^{-}:\CC\to[0,1]$
will be functions that are identically zero in a neighborhood of $\{1\}\times S^{1}$,
identically equal to $1$ in a neighborhood of $\{\infty\}\times S^{1}$,
and such that 
\begin{equation}
\chi_{\CC}^{+}\chi_{\CC}^{-}=\chi_{\CC}^{-}.\label{eq:stagger}
\end{equation}
We will extend both these functions by zero to functions on $X$,
and define
\[
\chi_{\CC,\phi}^{\pm}\eqdf\chi_{\CC}^{\pm}\circ\pi_{\phi}:X_{\phi}\to[0,1]
\]
where $\pi_{\phi}:X_{\phi}\to X$ is the covering map. Indeed, the
cusp of $X$ splits in $X_{\phi}$ into several regions of the form
\begin{equation}
(1,\infty)\times\left(\R/m\Z\right),\label{eq:covering-cusp}
\end{equation}
with $m\in\N$, and with the same metric (\ref{eq:cusp-metric}).
In these coordinates the covering map sends 
\[
\pi_{\phi}:(r,x+m\Z)\mapsto(r,x+\Z).
\]
In particular, it preserves the $r$ coordinate.
\begin{lem}
\label{lem:cusp-cutoff}For any $\epsilon>0$, it is possible to choose
$\chi_{\CC}^{\pm}$ as above so that for any $\phi$
\[
\|\nabla\chi_{\CC,\phi}^{+}\|_{\infty},\|\Delta\chi_{\CC,\phi}^{+}\|_{\infty}\leq\epsilon.
\]
\end{lem}

\begin{proof}
Given $\epsilon>0$, let $\chi_{\CC}^{+}:[0,\infty)\to[0,1]$ be a
function such that $\chi_{\CC}^{+}(\tau)\equiv0$ for $\tau$ in $[0,1]$,
$\chi_{\CC}^{+}(\tau)\equiv1$ for $\tau\geq\tau_{0}$, for some $\tau_{0}>0$,
and such that we have the following bounds for the derivatives of
$\chi_{\CC}^{+}$
\[
\sup_{[0,\infty)}|(\chi_{\CC}^{+})'|,\,\sup_{[0,\infty)}|(\chi_{\CC}^{+})''|\leq\frac{\epsilon}{2}
\]
(this can be achieved by scaling some fixed cutoff).

Let $\CC'$ be any cusp region of $X_{\phi}$ as in (\ref{eq:covering-cusp}).
Using the change of coordinates $r=e^{\tau}$ we view $\CC'$ as
\[
(0,\infty)_{\tau}\times\R/m\Z
\]
 with the metric $(d\tau)^{2}+e^{-2\tau}(dx)^{2}$; $x$ being the
coordinate in $\R/m\Z$. In these coordinates, one calculates directly
from the formula for the metric that
\begin{align*}
\|\nabla\chi_{\CC,\phi}^{+}\|(\tau,x+m\Z) & =|[\chi_{\CC}^{+}]'(\tau)|\leq\epsilon
\end{align*}
 and 
\begin{align*}
|\Delta\chi_{\CC,\phi}^{+}|(\tau,x+m\Z) & =|[\chi_{\CC}^{+}]''(\tau)-[\chi_{\CC}^{+}]'(\tau)|\leq\epsilon
\end{align*}
 as required. Finally, one can easily choose $\chi_{\CC}^{-}$ to
be a function with $\chi_{\CC}^{-}(\tau)$ $\equiv0$ for $\tau\leq\tau_{0}$
and $\chi_{\CC}^{-}(\tau)\equiv1$ for $\tau\geq2\tau_{0}$. This
will fulfill (\ref{eq:stagger}).
\end{proof}
Let $\CC_{\phi}$ denote the subset of $X_{\phi}$ that covers $\CC$.
It is convenient, to avoid complicated discussions about Sobolev spaces,
to extend $\CC$ to the parabolic cylinder
\[
\tilde{\CC}\eqdf(0,\infty)\times S^{1}
\]
 with the same metric (\ref{eq:cusp-metric}), and let $\tilde{\CC}_{\phi}$
be the corresponding extension of $\CC_{\phi}$ (the union of extensions
of (\ref{eq:covering-cusp}) to $(0,\infty)\times\R/m\Z$). 

Let $H^{2}(\tilde{\CC}_{\phi})$ denote the completion of $C_{c}^{\infty}(\tilde{\CC}_{\phi})$
with respect to the given norm
\[
\|f\|_{H^{2}}^{2}\eqdf\|f\|_{L^{2}}^{2}+\|\Delta f\|_{L^{2}}^{2}.
\]
The Laplacian $\Delta=\Delta_{\tilde{\CC}_{\phi}}$ extends uniquely
from $C_{c}^{\infty}(\tilde{\CC}_{\phi})$ to a self-adjoint unbounded
operator on $L^{2}(\tilde{\CC}_{\phi})$ with domain $H^{2}(\tilde{\CC}_{\phi})$. 
\begin{lem}
\label{lem:bounded-below}For any $f\in H^{2}(\tilde{\CC}_{\phi})$,
we have $\langle\Delta f,f\rangle\geq\frac{1}{4}\|f\|^{2}.$
\end{lem}

\begin{proof}
This is similar to \cite[Lemma 3.2]{Magee}. It suffices to prove
this for $\tilde{\CC}_{\phi}$ replaced by $(0,\infty)\times(\R/m\Z)$
with the metric (\ref{eq:cusp-metric}) i.e. with only one connected
component. Then changing coordinates to $\tau$ we are working in
the region $(-\infty,\infty)\times(\R/m\Z)$ with the metric $(d\tau)^{2}+e^{-2\tau}(dx)^{2}$.
The corresponding volume form is $e^{-\tau}d\tau\wedge dx$ and the
Laplacian is given by $\Delta=-e^{\tau}\frac{\partial}{\partial\tau}e^{-\tau}\frac{\partial}{\partial\tau}-e^{2\tau}\frac{\partial^{2}}{\partial\theta^{2}}$.
Now suppose $f\in C_{c}^{\infty}((-\infty,\infty)\times(\R/m\Z)).$
We calculate
\[
e^{-\tau/2}\Delta e^{\tau/2}=-\frac{\partial^{2}}{\partial\tau^{2}}+\frac{1}{4}-e^{2\tau}\frac{\partial^{2}}{\partial\theta^{2}}
\]
so if $g=e^{-\tau/2}f\in C_{c}^{\infty}((-\infty,\infty)\times(\R/m\Z))$
we have 
\begin{align*}
\int\Delta[f]\bar{f}e^{-\tau}d\tau\wedge dx & =\int_{-\infty}^{\infty}\int_{0}^{n}[e^{-\tau/2}\Delta e^{\tau/2}]\left(g\right)\bar{g}d\tau\wedge dx\\
 & \geq\frac{1}{4}\int_{-\infty}^{\infty}\int_{0}^{n}g\bar{g}d\tau\wedge dx=\frac{1}{4}\int_{-\infty}^{\infty}\int_{0}^{n}f\bar{f}e^{-\tau}d\tau\wedge dx.
\end{align*}
The inequality here used integrating by parts. The inequality obtained
now extends to $H^{2}(\tilde{\CC}_{\phi})$ by density of $C_{c}^{\infty}(\tilde{C}_{\phi})$
therein and continuity of $\langle\Delta f,f\rangle$.
\end{proof}
Lemma \ref{lem:bounded-below} implies that the resolvent operator
\[
R_{\tilde{\CC}_{\phi}}(s)\eqdf(\Delta-s(1-s))^{-1}:L^{2}(\tilde{\CC}_{\phi})\to H^{2}(\tilde{\CC}_{\phi})
\]
 is a holomorphic family of bounded operators in $\Re(s)>\frac{1}{2}$,
each a bijection to their image. This gives an a priori bound for
the resolvent: using 
\[
\left(\Delta-s(1-s)\right)R_{\tilde{\CC}_{\phi}}(s)f=f
\]
and Lemma \ref{lem:bounded-below} we obtain that for $f\in L^{2}(\tilde{\CC}_{\phi})$
and $s\in(\frac{1}{2},\infty)$
\begin{align}
\|R_{\tilde{\CC}_{\phi}}(s)f\|_{L^{2}} & \leq\left(\frac{1}{4}-s(1-s)\right)^{-1}\|f\|_{L^{2}}\label{eq:L2-resolvent}
\end{align}
 and as
\[
\Delta R_{\tilde{\CC}_{\phi}}(s)f=f+s(1-s)R_{\tilde{\CC}_{\phi}}(s)f
\]
 we obtain for $s\in(\frac{1}{2},\infty)$
\begin{align}
\|\Delta R_{\tilde{\CC}_{\phi}}(s)f\|_{L^{2}} & \leq\|f\|_{L^{2}}+s(1-s)\|R_{\tilde{\CC}_{\phi}}(s)f\|_{L^{2}}\nonumber \\
 & \leq\left(1+\frac{s(1-s)}{\frac{1}{4}-s(1-s)}\right)\|f\|_{L^{2}}\nonumber \\
 & =\frac{1}{1-4s(1-s)}\|f\|_{L^{2}}.\label{eq:secondH2-resolvent}
\end{align}
We now define the cusp parametrix as 
\begin{equation}
\M_{\phi}^{\cusp}(s)\eqdf\chi_{\CC,\phi}^{+}R_{\tilde{\CC}_{\phi}}(s)\chi_{\CC,\phi}^{-}.\label{eq:mcusp-def}
\end{equation}
Here, 
\begin{itemize}
\item (multiplication by) $\chi_{\CC,\phi}^{-}$ is viewed as an operator
from $L^{2}(X_{\phi})$ to $L^{2}(\tilde{\CC}_{\phi})$ by mapping
first to $L^{2}(\CC_{\phi})$ and then extending by zero. This is
a bounded linear operator.
\item $R_{\tilde{\CC}_{\phi}}(s)$ is a bounded operator from $L^{2}(\tilde{\CC}_{\phi})$
to $H^{2}(\tilde{\CC}_{\phi})$.
\item (multiplication by) $\chi_{\CC,\phi}^{+}$ is viewed as a operator
from $H^{2}(\tilde{\CC}_{\phi})$ to $H^{2}(X_{\phi})$, using that
$\chi_{\CC,\phi}^{+}$ localizes to $\CC_{\phi}$ and then extending
by zero. This operator is bounded because derivatives of $\chi_{\CC,\phi}^{+}$
are bounded and compactly supported.
\end{itemize}
Hence 
\[
\mathbb{M}_{\phi}^{\cusp}(s):L^{2}(X_{\phi})\to H^{2}(X_{\phi})
\]
is a bounded operator.

The covering map $\CC_{\phi}\to\CC$ extends in an obvious way to
a covering map $\tilde{\CC}_{\phi}\to\tilde{\CC}$ that intertwines
the two Laplacian operators. This, together with the fact that multiplication
by $\chi_{\CC,\phi}^{-}$ and $\chi_{\CC,\phi}^{+}$ leave invariant
the subspaces of functions lifted through the covering map, one sees
that
\[
\M_{\phi}^{\cusp}(s)(L_{\mathrm{new}}^{2}(X_{\phi}))\subset H_{\new}^{2}(X_{\phi}).
\]
Because $\chi_{\CC,\phi}^{+}\chi_{\CC,\phi}^{-}=\chi_{\CC,\phi}^{-}$,
\begin{equation}
(\Delta-s(1-s))\M_{\phi}^{\cusp}(s)=\chi_{\CC,\phi}^{-}+[\Delta,\chi_{\CC,\phi}^{+}]R_{\tilde{\CC}_{\phi}}(s)\chi_{\CC,\phi}^{-}=\chi_{\CC,\phi}^{-}+\LL_{\phi}^{\cusp}(s)\label{eq:cusp-parametrix-equation}
\end{equation}
 where 
\[
\LL_{\phi}^{\cusp}(s)\eqdf[\Delta,\chi_{\CC,\phi}^{+}]R_{\tilde{\CC}_{\phi}}(s)\chi_{\CC,\phi}^{-}
\]
and $[A,B]\eqdf AB-BA$ denotes the commutator of linear maps. Here
again we view $\chi_{\CC,\phi}^{-}$ and $R_{\tilde{\CC}_{\phi}}(s)$
as above, and $[\Delta,\chi_{\CC,\phi}^{+}]:H^{2}(\tilde{\CC}_{\phi})\to L^{2}(X_{\phi})$.
This means that $\LL_{\phi}^{\cusp}(s)$ is an operator on $L^{2}(X_{\phi})$.
By similar arguments to before, using that $[\Delta,\chi_{\CC,\phi}^{+}]$
only involves radial derivatives (since $\chi_{\CC,\phi}^{+}$ is
radial), we obtain 
\[
\LL_{\phi}^{\cusp}(s)\left(L_{\mathrm{new}}^{2}(X_{\phi})\right)\subset L_{\mathrm{new}}^{2}(X_{\phi}).
\]

\begin{lem}
\label{lem:cusp-error-bound-init}Given $s_{0}>\frac{1}{2}$, there
exists $C(s_{0})>0$ such that for $s\in[s_{0},1]$, the operator
$\LL_{\phi}^{\cusp}(s)$ is a self-adjoint, bounded operator on $L^{2}(X_{\phi})$
with operator norm
\[
\|\LL_{\phi}^{\cusp}(s)\|_{L^{2}}\leq\left(\|(\Delta\chi_{\CC,\phi}^{+})\|_{\infty}+2\|\nabla\chi_{\CC,\phi}^{+}\|_{\infty}\right)C(s_{0}).
\]
\end{lem}

\begin{proof}
As an operator on $H^{2}(\tilde{\CC}_{\phi})$
\[
[\Delta,\chi_{\CC,\phi}^{+}]=(\Delta\chi_{\CC,\phi}^{+})-2(\nabla\chi_{\CC,\phi}^{+})\cdot\nabla.
\]
The first summand is a multiplication operator; for $f\in H^{2}(\tilde{\CC}_{\phi})$
we have 
\begin{equation}
\|(\Delta\chi_{\CC,\phi}^{+})f\|_{L^{2}}\leq\|(\Delta\chi_{\CC,\phi}^{+})\|_{\infty}\|f\|_{L^{2}}\label{eq:temp1}
\end{equation}
and by Schwarz inequality if $\|f\|_{H^{2}}\leq1$ then 
\begin{align}
\|(\nabla\chi_{\CC,\phi}^{+})\cdot\nabla f\|_{L^{2}} & \leq\|\nabla\chi_{\CC,\phi}^{+}\|_{\infty}\|\nabla f\|_{L^{2}}\nonumber \\
 & =\|\nabla\chi_{\CC,\phi}^{+}\|_{\infty}\langle\Delta f,f\rangle^{\frac{1}{2}}\leq\|\nabla\chi_{\CC,\phi}^{+}\|_{\infty}\|\Delta f\|^{\frac{1}{2}}\|f\|^{\frac{1}{2}}\nonumber \\
 & \leq\|\nabla\chi_{\CC,\phi}^{+}\|_{\infty}.\label{eq:temp1-1}
\end{align}
The two estimates (\ref{eq:temp1}), (\ref{eq:temp1-1}) hence show
that $[\Delta,\chi_{\CC,\phi}^{+}]$ has norm bounded by $\|(\Delta\chi_{\CC,\phi}^{+})\|_{\infty}+2\|\nabla\chi_{\CC,\phi}^{+}\|_{\infty}$
as a map $H^{2}(\tilde{\CC}_{\phi})\to L^{2}(X_{\phi})$. Since multiplication
by $\chi_{\CC,\phi}^{-}$ has norm $\leq1$ from $L^{2}$ to $L^{2}$,
and $R_{\tilde{\CC}_{\phi}}(s)$ has norm $\leq C(s_{0})$ from $L^{2}(\tilde{\CC}_{\phi})$
to $H^{2}(\tilde{\CC}_{\phi})$ by (\ref{eq:L2-resolvent}) and (\ref{eq:secondH2-resolvent}),
we obtain the stated result.
\end{proof}
We can now obtain the following key proposition by combining Lemmas
\ref{lem:cusp-cutoff} and \ref{lem:cusp-error-bound-init}.
\begin{prop}
\label{prop:L-cusp-bound}Given any $s_{0}>\frac{1}{2}$, we can choose
$\chi_{\CC}^{+}$ and $\chi_{\CC}^{-}$ (depending on $s_{0})$ such
that for any $n\in\N$ and $\phi\in\Hom(\Gamma,S_{n})$ and any $s\in[s_{0},1]$,
we have for the operator norm of $\LL_{\phi}^{\cusp}(s)$ on $L^{2}(X_{\phi})$
\[
\|\LL_{\phi}^{\cusp}(s)\|_{L^{2}}\leq\frac{1}{5}.
\]
\end{prop}

\emph{We assume that $\chi_{\CC}^{-}$ and $\chi_{\CC}^{+}$ have
now been fixed to satisfy Proposition \ref{prop:L-cusp-bound} in
the rest of the paper. Hence all constants may depend on these functions.}

\section{Interior parametrix}

\subsection{Background on spectral theory of hyperbolic plane\label{subsec:Background-on-spectral}}

\subsubsection*{Resolvent}

Let 
\begin{align*}
R_{\H}(s) & :L^{2}(\H)\to L^{2}(\H),\\
R_{\H}(s) & \eqdf(\Delta_{\H}-s(1-s))^{-1}.
\end{align*}
Since $\spec(\Delta_{\H})\subset[\frac{1}{4},\infty$) \cite[Thm. 4.3]{Borthwick},
this is well defined as a bounded operator for $\Re(s)>\frac{1}{2}$. 

Given $x,y\in\H$, we write $r=d_{\H}(x,y)$ where $d_{\H}$ is hyperbolic
distance and 
\[
\sigma\eqdf\cosh^{2}\left(\frac{r}{2}\right).
\]
The operator $R_{\H}(s)$ is an integral operator with kernel \cite[Prop. 4.2]{Borthwick}
\begin{equation}
R_{\H}(s;x,y)=\frac{1}{4\pi}\int_{0}^{1}\frac{t^{s-1}(1-t)^{s-1}}{(\sigma-t)^{s}}dt.\label{eq:resolvent-expression}
\end{equation}

For $\sigma>1$ and $t\in(0,1)$ we have 
\begin{align*}
\frac{\partial}{\partial r}\frac{t^{s-1}(1-t)^{s-1}}{(\sigma-t)^{s}}= & -s\sinh\left(\frac{r}{2}\right)\cosh\left(\frac{r}{2}\right)\frac{t^{s-1}(1-t)^{s-1}}{(\sigma-t)^{s+1}},\\
\frac{\partial}{\partial s}\frac{t^{s-1}(1-t)^{s-1}}{(\sigma-t)^{s}}= & \log\left(\frac{t(1-t)}{(\sigma-t)}\right)\frac{t^{s-1}(1-t)^{s-1}}{(\sigma-t)^{s}},\\
\frac{\partial^{2}}{\partial s\partial r}\frac{t^{s-1}(1-t)^{s-1}}{(\sigma-t)^{s}}= & -\sinh\left(\frac{r}{2}\right)\cosh\left(\frac{r}{2}\right)\frac{t^{s-1}(1-t)^{s-1}}{(\sigma-t)^{s+1}}\\
 & \cdot\left[1+s\log\left(\frac{t(1-t)}{(\sigma-t)}\right)\right].
\end{align*}
Each of these are smooth in $(s,r,t)\in[\frac{1}{2},1]\times[1,\infty)\times(0,1)$.
Because for $s,r$ in a fixed compact set of $[\frac{1}{2},1]\times[1,\infty)$,
these all have absolute values majorized by integrable functions of
$t\in(0,1)$, we can interchange derivatives and integrals to get
corresponding bounds for $R_{\H}$ as follows. We define
\[
R_{\H}(s;r)\eqdf R_{\H}(s;x,y).
\]
Firstly, there is a constant $C>0$ such that for $r_{0}\geq1$ and
$s\in[\frac{1}{2},1]$we have 
\begin{equation}
\left|R_{\H}(s;r_{0})\right|,\left|\frac{\partial R_{\H}}{\partial r}(s;r_{0})\right|\leq Ce^{-sr_{0}}.\label{eq:resolvent-bounds-plane-r-deriv}
\end{equation}
Secondly, for any $T>1$ and $r_{0}$ in the compact region $[1,T+1]$
we obtain for constant $c=c(T)>0$, for all $s_{0}\in[\frac{1}{2},1]$,
\begin{align}
\left|\frac{\partial R_{\H}}{\partial s}(s_{0};r_{0})\right|,\left|\frac{\partial^{2}R_{\H}}{\partial s\partial r}(s_{0};r_{0})\right| & \leq c(T).\label{eq:resolvent-bounds-plane-s-deriv}
\end{align}

\subsubsection*{Integral operators}

If $k_{0}:[0,\infty)\to\R$ is smooth and compactly supported, which
will suffice here, then one can construct a kernel
\[
k(x,y)\eqdf k_{0}(d_{\H}(x,y))
\]
with corresponding integral operator $C^{\infty}(\H)\to C^{\infty}(\H)$
\[
K[f](x)\eqdf\int_{y\in\H}k(x,y)f(y)d\H(y)
\]
where $d\H$ is the hyperbolic area form on $\H$. Such an operator
commutes with the Laplacian on $\H$ and hence preserves its generalized
eigenspaces. If $f\in C^{\infty}(\H)$ is a generalized eigenfunction
of $\Delta$ with eigenvalue $\frac{1}{4}+\xi^{2}$, $\xi\geq0$,
then by \cite[Thm. 3.7, Lemma 3.9]{BergeronBook} (cf. also Selberg's
original article \cite{Selberg})
\[
K[f]=h(\xi)f
\]
where 
\[
h(\xi)=\sqrt{2}\int_{-\infty}^{\infty}e^{i\xi u}\int_{|u|}^{\infty}\frac{k_{0}(\rho)\sinh(\rho)}{\sqrt{\cosh(\rho)-\cosh(u)}}d\rho du.
\]
By our assumptions on $k_{0}$ the integral above is convergent. It
will be estimated more precisely for particular choice of $k_{0}$
in the proof of Lemma \ref{lem:LH-noorm-bound}. For now note that
if $k_{0}$ is real valued then $h$ is the Fourier transform of a
real valued even function and hence real-valued. Since $L^{2}(\H)$
has a generalized eigenbasis of $C^{\infty}$ eigenfunctions of the
Laplacian, by Borel functional calculus $K$ extends from e.g. $C_{c}^{\infty}(\H)$
to a (possibly unbounded) self-adjoint operator on $L^{2}(\H)$ with
operator norm
\begin{equation}
\|K\|_{L^{2}(\H)}=\sup_{\xi\geq0}|h(\xi)|.\label{eq:operator-norm-formula}
\end{equation}

\subsection{Interior parametrix on hyperbolic plane\label{subsec:Interior-parametrix-on-plane}}

Let $\chi_{0}:\R\to[0,1]$ denote a smooth function with $\chi_{0}\lvert_{(-\infty,0]}\equiv1$
and $\chi_{0}\lvert_{[1,\infty)}=0$. Then let 
\[
\chi_{T}(t)\eqdf\chi_{0}(t-T)
\]
so $\chi_{T}$ is supported on $(-\infty,T+1]$; note for later that
all derivatives of $\chi_{T}$ are supported in $[T,T+1]$ and have
uniform bounds independent of $T$.

We define
\[
R_{\H}^{(T)}(s;r)\eqdf\chi_{T}(r)R_{\H}(s;r).
\]
Let $R_{\H}^{(T)}(s)$ denote the corresponding integral operator.
In radial coordinates the Laplacian on $\H$ is given by \cite[pg. 65]{Borthwick}
\[
-\frac{\partial^{2}}{\partial r^{2}}-\frac{1}{\tanh r}\frac{\partial}{\partial r}-\frac{1}{\sinh^{2}r}\frac{\partial^{2}}{\partial\theta^{2}}.
\]

We now perform the following calculation, writing $\Delta_{x}$ for
the Laplacian acting on coordinate $x$:
\begin{align}
\left[\Delta_{x}-s(1-s)\right]R_{\H}^{(T)}(s;r) & =\left[\Delta_{x}-s(1-s)\right]\left(\chi_{T}(r)R_{\H}(s;r)\right)\nonumber \\
 & =\left[-\frac{\partial^{2}}{\partial r^{2}}-\frac{1}{\tanh r}\frac{\partial}{\partial r},\chi_{T}\right]R_{\H}(s;r)+\delta_{r=0}\label{eq:h-para-1}
\end{align}
which is understood in a distributional sense. We further calculate
\begin{align}
\left[-\frac{\partial^{2}}{\partial r^{2}}-\frac{1}{\tanh r}\frac{\partial}{\partial r},\chi_{T}\right] & =-\frac{\partial^{2}}{\partial r^{2}}[\chi_{T}]-2\frac{\partial}{\partial r}[\chi_{T}]\frac{\partial}{\partial r}-\frac{1}{\tanh r}\frac{\partial}{\partial r}[\chi_{T}].\label{eq:h-para-2}
\end{align}
Combining (\ref{eq:h-para-1}) and (\ref{eq:h-para-2}) we expect
an identity of operators
\begin{equation}
\left[\Delta-s(1-s)\right]R_{\H}^{(T)}(s)=1+\LL_{\H}^{(T)}(s)\label{eq:interior-parametrix-identity-H}
\end{equation}
where $\LL_{\H}^{(T)}(s)$ is the integral operator with real-valued
radial kernel
\begin{equation}
\LL_{\H}^{(T)}(s;r_{0})\eqdf\left(-\frac{\partial^{2}}{\partial r^{2}}[\chi_{T}]-\frac{1}{\tanh r}\frac{\partial}{\partial r}[\chi_{T}]\right)R_{\H}(s;r_{0})-2\frac{\partial}{\partial r}[\chi_{T}]\frac{\partial R_{\H}}{\partial r}(s;r_{0}).\label{eq:remainder-kernel-def}
\end{equation}
The identity (\ref{eq:interior-parametrix-identity-H}) will be established
in Lemma \ref{lem:sobolev-on-H} below.

We note the following properties of this kernel that are straightforward
consequences of the way $\chi_{T}$ was chosen and (\ref{eq:resolvent-bounds-plane-r-deriv}),
(\ref{eq:resolvent-bounds-plane-s-deriv}).
\begin{lem}
\label{lem:LH-bounds}We have 
\begin{enumerate}
\item For $T>0$ and $s\in[\frac{1}{2},1]$, $\LL_{\H}^{(T)}(s;\bullet)$
is smooth and supported in $[T,T+1]$.
\item There is a constant $C>0$ such that for any $T>0$ and $s\in[\frac{1}{2},1]$
we have 
\[
|\LL_{\H}^{(T)}(s;r_{0})|\leq Ce^{-sr_{0}}
\]
\item \label{enu:s-deriv}For any $T>0$ there is a constant $c(T)>0$ such
that for any $r_{0}\in[T,T+1]$
\[
\left|\frac{\partial\LL_{\H}^{(T)}}{\partial s}(s_{0};r_{0})\right|\leq c(T).
\]
\end{enumerate}
\end{lem}

We can now show the following.
\begin{lem}
\label{lem:LH-noorm-bound}There is a constant $C>0$ such that for
any $T>0$ and $s\in[\frac{1}{2},1]$ the operator $\LL_{\H}^{(T)}(s)$
extends to a bounded operator on $L^{2}(\H)$ with operator norm

\[
\|\LL_{\H}^{(T)}(s)\|_{L^{2}}\leq CTe^{\left(\frac{1}{2}-s\right)T}.
\]
\end{lem}

\begin{proof}
We do this using (\ref{eq:operator-norm-formula}) which tells us
\begin{align*}
\|\LL_{\H}^{(T)}(s)\|_{L^{2}} & =\sup_{\xi\geq0}\left|\sqrt{2}\int_{-\infty}^{\infty}e^{i\xi u}\int_{|u|}^{\infty}\frac{\mathbb{L}_{\H}^{(T)}(s;\rho)\sinh(\rho)}{\sqrt{\cosh(\rho)-\cosh(u)}}d\rho du\right|\\
 & \leq\sqrt{2}\int_{-\infty}^{\infty}\int_{|u|}^{\infty}\frac{|\mathbb{L}_{\H}^{(T)}(s;\rho)|\sinh(\rho)}{\sqrt{\cosh(\rho)-\cosh(u)}}d\rho du\\
 & \leq2\sqrt{2}C\int_{0}^{T+1}\int_{\max(|u|,T)}^{T+1}\frac{e^{-s\rho}\sinh(\rho)}{\sqrt{\cosh(\rho)-\cosh(u)}}d\rho du\\
 & \leq C'e^{-sT}\int_{0}^{T+1}\int_{\max(|u|,T)}^{T+1}\frac{\sinh(\rho)}{\sqrt{\cosh(\rho)-\cosh(u)}}d\rho du\\
 & =C'e^{-sT}\int_{0}^{T+1}\int_{\cosh\max(|u|,T)}^{\cosh(T+1)}\frac{dy}{\sqrt{y-\cosh(u)}}du\\
 & =C''e^{-sT}\int_{0}^{T+1}\Big[\sqrt{\cosh(T+1)-\cosh u}\\
 & \quad\quad\quad\quad\quad\quad-\sqrt{\cosh\max(|u|,T)-\cosh|u|}\Big]du\\
 & \leq C'''Te^{\left(\frac{1}{2}-s\right)T}
\end{align*}
 where the third inequality used Lemma \ref{lem:LH-bounds}.
\end{proof}
The implication of this is that for any $s_{0}>\frac{1}{2}$ we can
choose $T=T(s_{0})$ such that for all $s\in[s_{0},1]$ we have 
\[
\|\LL_{\H}^{(T)}(s)\|_{L^{2}}\leq\frac{1}{5}.
\]
We will do this later. The following lemma shows that smoothly cutting
off $R_{\H}(s)$ at radius $T$ does not significantly affect its
mapping properties.
\begin{lem}
\label{lem:sobolev-on-H}For any $T>0$ and $s\in[\frac{1}{2},1]$,
for any compact $K\subset\H$, there is $C=C(s,K,T)>0$ such that:
\begin{enumerate}
\item \label{enu:L^2} For any $f\in C_{c}^{\infty}(\H)$ with $\mathrm{supp}(f)\subset K$
we have $R_{\H}^{(T)}(s)f\in H^{2}(\H)$ and 
\[
\|R_{\H}^{(T)}(s)f\|_{H^{2}}\leq C(s,K,T)\|f\|_{L^{2}}.
\]
\item \label{enu:functional-identity}Furthermore, with $f$ as above
\[
(\Delta-s(1-s))R_{\H}^{(T)}(s)[f]=f+\LL_{\H}^{(T)}(s)[f]
\]
in the sense of equivalence of $L^{2}$ functions.
\end{enumerate}
\end{lem}

\begin{proof}
Suppose that compact $K$ is given and $f\in C_{c}^{\infty}(\H)$
with $\mathrm{supp}(f)\subset K$. For $y\in K$ we have $R_{\H}^{(T)}(s;x,y)=0$
unless 
\[
x\in K'(T,K)\eqdf\{\,x\,:\,d(x,K)\leq T+1\,\}
\]
with $K'$ compact. Therefore using the usual Hilbert-Schmidt inequality
we obtain
\begin{align}
 & \|R_{\H}^{(T)}(s)[f]\|_{L^{2}(\H)}^{2}\nonumber \\
= & \int_{x\in K'}\left|\int_{y\in K}R_{\H}^{(T)}(s;x,y)f(y)d\H(y)\right|^{2}d\H(x)\nonumber \\
\leq & \int_{x\in K'}\left(\int_{y\in K}R_{\H}^{(T)}(s;x,y)^{2}d\H(y)\right)\left(\int_{y\in K}|f(y)|^{2}d\H(y)\right)d\H(x).\label{eq:hs-ineq}
\end{align}
Recall that we write $r=d_{\H}(x,y)$, hence the inner integral can
be written in polar coordinates as
\begin{align}
\int_{y\in K}R_{\H}^{(T)}(s;x,y)^{2}d\H(y) & =\int_{0}^{2\pi}\int_{0}^{\infty}R_{\H}^{(T)}(s;r)^{2}\sinh r\,dr\,d\theta\nonumber \\
 & \leq2\pi\int_{0}^{M}R_{\H}^{(T)}(s;r)^{2}\sinh r\,dr\label{eq:l2bound2-1}
\end{align}
for $M=M(K,T)$. Because $\chi_{T}\equiv1$ near $0$, the type of
singularity that $R_{\H,n}^{(T)}(s;r)$ has at $r=0$ is exactly the
same as the type of singularity of $R_{\H}(s;r)$ near $r=0$; namely
by \cite[(4.2)]{Borthwick}
\begin{equation}
R_{\H}^{(T)}(s;r)=-\frac{1}{4\pi}\log\left(\frac{r}{2}\right)+O(1)\label{eq:diagonal-asymp}
\end{equation}
as $r\to0$. The function $R_{\H}^{(T)}(s;r)$ is smooth away from
0. Hence, since $\left(\log\left(\frac{r}{2}\right)\right)^{2}\sinh r\to0$
as $r\to0$, $R_{\H}^{(T)}(s;r)$ is in particular square integrable
on $\left[0,M\right]$. This gives from (\ref{eq:hs-ineq})
\begin{equation}
\|R_{\H}^{(T)}(s)[f]\|_{L^{2}(\H)}^{2}\leq\int_{x\in K'}C(s,K,T)\|f\|_{L^{2}(\H)}^{2}d\H(x)\leq C'(s,K,T)\|f\|_{L^{2}(\H)}^{2}.\label{eq:l2-bound}
\end{equation}

We now aim for a bound on $\|\Delta R_{\H}^{(T)}(s)[f]\|_{L^{2}(\H)}^{2}$
so as to prove $R_{\H}^{(T)}(s)[f]\in H^{2}(\H)$ and bound its $H^{2}$-norm.

Let $g\in C_{c}^{\infty}(\H)$ be a test function and $f\in C_{c}^{\infty}(\H)$
with support as above. Consider
\begin{align*}
\langle R_{\H}^{(T)}(s)[f],\Delta g\rangle & =\int_{x\in\H}\left(\int_{y\in\H}R_{\H}^{(T)}(s;x,y)f(y)d\H(y)\right)\Delta\overline{g}(x)d\H(x).
\end{align*}
Because $f$ and $g$ are compactly supported and the singularity
of $R_{\H}^{(T)}(x,y)$ is locally $L^{1}$ using (\ref{eq:diagonal-asymp}),
we can use Fubini to get
\begin{equation}
\langle R_{\H}^{(T)}(s)[f],\Delta g\rangle=\int_{y\in\H}\left(\int_{x\in\H}R_{\H}^{(T)}(s;x,y)\Delta\overline{g}(x)d\H(x)\right)f(y)d\H(y).\label{eq:start}
\end{equation}
We use hyperbolic polar coordinates for the inner integral, writing
$r=d(x,y)$ and $\theta$ for polar angle, $G_{y}(r,\theta)\eqdf\bar{g}(x)$,
and the inner integral is understood as an improper integral as follows:
\begin{align}
 & \int_{x\in\H}R_{\H}^{(T)}(s;x,y)\Delta\overline{g}(x)d\H(x)\nonumber \\
= & -\lim_{\epsilon\to0}\int_{\epsilon}^{\infty}\int_{0}^{2\pi}R_{\H}^{(T)}(s;\tilde{r})\left(\frac{\partial}{\partial r}\left(\sinh r\frac{\partial G_{y}}{\partial r}\right)(\tilde{r},\tilde{\theta})+\frac{1}{\left(\sinh r\right)^{2}}\frac{\partial^{2}G_{y}}{\partial\theta^{2}}(\tilde{r},\tilde{\theta})\right)d\tilde{\theta}d\tilde{r}\nonumber \\
= & -\lim_{\epsilon\to0}\int_{\epsilon}^{\infty}\int_{0}^{2\pi}R_{\H}^{(T)}(s;\tilde{r})\left(\frac{\partial}{\partial r}\left(\sinh r\frac{\partial G_{y}}{\partial r}\right)(\tilde{r},\tilde{\theta})\right)d\tilde{\theta}d\tilde{r}\nonumber \\
= & -\lim_{\epsilon\to0}\int_{0}^{2\pi}\int_{\epsilon}^{\infty}R_{\H}^{(T)}(s;\tilde{r})\left(\frac{\partial}{\partial r}\left(\sinh r\frac{\partial G_{y}}{\partial r}\right)(\tilde{r},\tilde{\theta})\right)d\tilde{r}d\tilde{\theta}\nonumber \\
= & \lim_{\epsilon\to0}R_{\H}^{(T)}(s;\epsilon)\sinh\epsilon\int_{0}^{2\pi}\frac{\partial G_{y}}{\partial r}(\epsilon,\tilde{\theta})d\tilde{\theta}\nonumber \\
 & \quad+\lim_{\epsilon\to0}\int_{0}^{2\pi}\int_{\epsilon}^{\infty}\frac{\partial R_{\H}^{(T)}}{\partial r}(s;\tilde{r})\sinh\tilde{r}\frac{\partial G_{y}}{\partial r}(\tilde{r},\tilde{\theta})d\tilde{r}d\tilde{\theta}\nonumber \\
= & \lim_{\epsilon\to0}\int_{0}^{2\pi}\int_{\epsilon}^{\infty}\frac{\partial R_{\H}^{(T)}}{\partial r}(s;\tilde{r})\sinh\tilde{r}\frac{\partial G_{y}}{\partial r}(\tilde{r},\tilde{\theta})d\tilde{r}d\tilde{\theta}\label{eq:H2-temp1}
\end{align}
where the last equality used (\ref{eq:diagonal-asymp}) with smoothness
of $G_{y}$. Now a similar calculation gives
\begin{align}
 & \lim_{\epsilon\to0}\int_{0}^{2\pi}\int_{\epsilon}^{\infty}\frac{\partial R_{\H}^{(T)}}{\partial r}(s;\tilde{r})\sinh r\frac{\partial G_{y}}{\partial r}(\tilde{r},\tilde{\theta})d\tilde{r}d\tilde{\theta}\nonumber \\
= & -\lim_{\epsilon\to0}\frac{\partial R_{\H}^{(T)}}{\partial r}(s;\epsilon)\sinh\epsilon\int_{0}^{2\pi}G_{y}(\epsilon,\tilde{\theta})d\tilde{\theta}\nonumber \\
 & -\lim_{\epsilon\to0}\int_{0}^{2\pi}\int_{\epsilon}^{\infty}\frac{\partial}{\partial r}\left(\sinh r\frac{\partial R_{\H}^{(T)}}{\partial r}\right)(s;\tilde{r})G_{y}(\tilde{r},\tilde{\theta})d\tilde{r}d\tilde{\theta}\nonumber \\
= & \bar{g}(y)-\lim_{\epsilon\to0}\int_{0}^{2\pi}\int_{\epsilon}^{\infty}\frac{\partial}{\partial r}\left(\sinh r\frac{\partial R_{\H}^{(T)}}{\partial r}\right)(s;\tilde{r})G_{y}(\tilde{r},\tilde{\theta})d\tilde{r}d\tilde{\theta}.\label{eq:H2-temp2}
\end{align}

The second equality used \cite[pg. 66]{Borthwick}
\[
\frac{\partial R_{\H}^{(T)}}{\partial r}(s;\epsilon)=-\frac{1}{2\pi\epsilon}+O(1)
\]
as $\epsilon\to0$ with $\lim_{\epsilon\to0}G_{y}(\epsilon,\tilde{\theta})=\bar{g}(y)$.
Now we note 
\[
-\frac{1}{\sinh r}\frac{\partial}{\partial r}\left(\sinh r\frac{\partial R_{\H}^{(T)}}{\partial r}\right)(s;\tilde{r})=\Delta R_{\H}^{(T)}(s;\tilde{r})
\]
and so using (\ref{eq:h-para-1}), (\ref{eq:h-para-2}) and (\ref{eq:remainder-kernel-def})
we get, for $\tilde{r}>0$,
\[
-\frac{1}{\sinh r}\frac{\partial}{\partial r}\left(\sinh r\frac{\partial R_{\H}^{(T)}}{\partial r}\right)(s;\tilde{r})=s(1-s)R_{\H}^{(T)}(s;\tilde{r})+\LL_{\H}^{(T)}(s;\tilde{r}).
\]
Therefore 
\begin{align}
 & -\lim_{\epsilon\to0}\int_{0}^{2\pi}\int_{\epsilon}^{\infty}\frac{\partial}{\partial r}\left(\sinh r\frac{\partial R_{\H}^{(T)}}{\partial r}\right)(s;\tilde{r})G_{y}(\tilde{r},\tilde{\theta})d\tilde{r}d\tilde{\theta}\nonumber \\
= & \lim_{\epsilon\to0}\int_{d(x,y)>\epsilon}\left(s(1-s)R_{\H}^{(T)}(s;\tilde{r})+\LL_{\H}^{(T)}(s;d(x,y))\right)\bar{g}(x)d\H(x)\nonumber \\
= & \int_{x\in\mathbb{H}}\left(s(1-s)R_{\H}^{(T)}(s;d(x,y))+\LL_{\H}^{(T)}(s;d(x,y))\right)\bar{g}(x)d\H(x)\label{eq:H2temp3}
\end{align}
and this last integral is easily seen to converge by working in polar
coordinates centered at $y$ and using $g\in C_{c}^{\infty}(\H)$
and (\ref{eq:diagonal-asymp}).

Now combining (\ref{eq:H2-temp1}), (\ref{eq:H2-temp2}), and (\ref{eq:H2temp3})
gives, for (\ref{eq:start}),
\begin{align*}
 & \langle R_{\H}^{(T)}(s)[f],\Delta g\rangle\\
 & =\int_{y\in\H}f(y)\bar{g}(y)d\H(y)\\
 & \,\,\,\,+\int_{y\in\mathbb{H}}\left(\int_{x\in\mathbb{H}}\left(s(1-s)R_{\H}^{(T)}(s;d(x,y))+\LL_{\H}^{(T)}(s;d(x,y))\right)\bar{g}(x)d\H(x)\right)f(y)d\H(y)\\
 & =\langle f,g\rangle+\langle s(1-s)R_{\H}^{(T)}(s)[f],g\rangle+\langle\LL_{\H}^{(T)}(s)[f],g\rangle.
\end{align*}
Note that by (\ref{eq:l2-bound}) and Lemma \ref{lem:LH-noorm-bound}
all functions above are in $L^{2}(\H)$. This identity now clearly
extends to any $g\in H^{2}(\H)$ and now self-adjointness of $\Delta_{\H}$
on $H^{2}(\H)$ (see \cite{Strichartz}) gives that $R_{\H}^{(T)}(s)[f]\in H^{2}(\H)$
and moreover 
\[
(\Delta-s(1-s))R_{\H}^{(T)}(s)[f]=f+\LL_{\H}^{(T)}(s)[f]
\]
in the sense of elements of $L^{2}(\H)$. \emph{This proves part 2
of the lemma.}

We now rewrite this identity as 
\[
\Delta R_{\H}^{(T)}(s)[f]=f+s(1-s)R_{\H}^{(T)}(s)[f]+\LL_{\H}^{(T)}(s)[f]
\]
and using Lemma \ref{lem:LH-noorm-bound} and (\ref{eq:l2-bound})
now gives
\[
\|\Delta R_{\H}^{(T)}(s)[f]\|_{L^{2}(\H)}\leq c(s,K,T)\|f\|_{L^{2}(\H)}.
\]
Combining this with (\ref{eq:l2-bound}), \emph{this proves part 1
of the lemma.}
\end{proof}

\subsection{Automorphic kernels}

At two points in the rest of the paper we use the following geometric
lemma. Recall that $F$ is a Dirichlet fundamental domain for $\Gamma$
in $\H$.
\begin{lem}
\label{lem:geometric}For any compact set $K\subset\bar{F}$, and
$T>0$, there is another compact set $K'=K'(K,T)\subset\bar{F}$ and
a \uline{finite} set $S=S(K,T)\subset\Gamma$ such that for all
$x,y\in\H$ with 
\[
x\in\bar{F},\,y\in\bigcup_{\gamma\in\Gamma}\gamma(K),\quad d(x,y)\leq T,
\]
we have 
\[
x\in K',\quad y\in\bigcup_{\gamma\in S}\gamma(K).
\]
\end{lem}

\begin{proof}
Suppose that $y=\gamma y'$ with $y'\in K$, $x\in\bar{F}$, $\gamma\in\Gamma$
and $d(x,y)\leq T$. Then $d(\gamma y',x)=d(y',\gamma^{-1}x)\leq T$
implies $\gamma^{-1}F$ intersects the compact set 
\[
\{\,z\,:\,d(z,K)\leq T\,\}.
\]
Since we took $F$ to be a Dirichlet fundamental domain, by \cite[Thm. 9.4.2]{Beardon}
there are only a finite number of $\gamma$ for which this is possible.
Let $S$ be this finite subset of $\Gamma$. Then furthermore, if
we let $o$ denote a fixed point in $K$ we have
\begin{align*}
d(x,o) & \leq d(x,\gamma y')+d(\gamma y',\gamma o)+d(\gamma o,o).\\
 & \leq T+d(y',o)+d(\gamma o,o)\leq C'(K,T).
\end{align*}
This implies that $x$ must be in some compact set $K'=K'(K,T)\subset\bar{F}$. 
\end{proof}
Define 
\begin{align*}
R_{\H,n}^{(T)}(s;x,y) & \eqdf R_{\H}^{(T)}(s;x,y)\mathrm{Id}_{V_{n}^{0}},\\
\LL_{\H,n}^{(T)}(s;x,y) & \eqdf\LL_{\H}^{(T)}(s;x,y)\mathrm{Id}_{V_{n}^{0}},
\end{align*}
and $R_{\H,n}^{(T)}(s),\LL_{\H,n}^{(T)}(s)$ the corresponding integral
operators. In the next lemma we describe the mapping properties of
these integral operators and describe a functional identity that they
satisfy that arises from (\ref{eq:interior-parametrix-identity-H}).
In the following we regard the function $\chi_{\CC}^{-}$ defined
in $\S$\ref{sec:Cusp-parametrix} as a $\Gamma$-invariant $C^{\infty}$
function on $\H$. The action of multiplication by $\chi_{\CC,\phi}^{-}$
on new functions in $L_{\new}^{2}(X_{\phi})$ corresponds to multiplication
by the scalar-valued invariant function $\chi_{\CC}^{-}$ on $L_{\phi}^{2}(\H;V_{n}^{0})$.
\begin{lem}
\label{lem:parametrix-bounded}For all $s\in[\frac{1}{2},1]$, 
\begin{enumerate}
\item The integral operator $R_{\H,n}^{(T)}(s)(1-\chi_{\CC}^{-})$ is well-defined
on $C_{c,\phi}^{\infty}(\H;V_{n}^{0})$ and extends to a bounded operator
\[
R_{\H,n}^{(T)}(s)(1-\chi_{\CC}^{-}):L_{\phi}^{2}(\H;V_{n}^{0})\to H_{\phi}^{2}(\H;V_{n}^{0}).
\]
\item The integral operator $\LL_{\H,n}^{(T)}(s)(1-\chi_{\CC}^{-})$ is
well-defined on $C_{c,\phi}^{\infty}(\H;V_{n}^{0})$ and and extends
to a bounded operator on $L_{\phi}^{2}(\H;V_{n}^{0})$.
\item We have 
\begin{equation}
\left[\Delta-s(1-s)\right]R_{\H,n}^{(T)}(s)(1-\chi_{\CC}^{-})=(1-\chi_{\CC}^{-})+\LL_{\H,n}^{(T)}(s)(1-\chi_{\CC}^{-})\label{eq:automorphic-para-eq}
\end{equation}
 as an identity of operators on $L_{\phi}^{2}(\H;V_{n}^{0})$.
\end{enumerate}
\end{lem}

\begin{proof}
Suppose first that $f\in C_{c,\phi}^{\infty}(\H;V_{n}^{0})$ (i.e.
automorphic, smooth, and compactly supported modulo $\Gamma$). We
have for $x\in F$
\begin{equation}
R_{\H,n}^{(T)}(s)(1-\chi_{\CC}^{-})[f](x)\eqdf\int_{y\in\H}R_{\H}^{(T)}(s;x,y)(1-\chi_{\CC}^{-}(y))f(y)d\H(y).\label{eq:truncated-resolvent-defining}
\end{equation}

The integrand here is non-zero unless $d(x,y)\leq T+1$ and $y$ is
in the the support of $1-\chi_{\CC}^{-}$, which is a union of the
$\Gamma$-translates of a compact set $K$ of of $\bar{F}$. Applying
Lemma \ref{lem:geometric} tells us then that there is compact $K_{1}=K_{1}(T)\subset\bar{F}$
and finite set $S=S(T)$ such that the integrand in (\ref{eq:truncated-resolvent-defining})
is supported on the compact set $K_{2}\eqdf\cup_{\gamma\in S}\gamma^{-1}K$
and the whole integral is zero unless $x\in K_{1}$ (given $x\in\bar{F}$
to begin with).

Let $\psi$ be a smooth function that is $\equiv1$ in $K_{2}\cup K$,
valued in $[0,1]$ and compactly supported. Let $\{e_{i}\,:\,i\in[n-1]\}$
denote an orthonormal basis for $V_{n}^{0}$ and let 
\[
f_{i}\eqdf\langle f,e_{i}\rangle\in C^{\infty}(\H).
\]
The above shows that for $x\in F$ we have
\begin{align}
R_{\H,n}^{(T)}(s)(1-\chi_{\CC}^{-})[f](x) & =R_{\H,n}^{(T)}(s)(1-\chi_{\CC}^{-})[\psi f](x)\nonumber \\
 & =\sum_{i=1}^{n-1}R_{\H}^{(T)}(s)\left[\left(1-\chi_{\CC}^{-}\right)\psi f_{i}\right](x)e_{i}\label{eq:psi-insert}
\end{align}
hence 
\begin{align*}
\|R_{\H,n}^{(T)}(s)(1-\chi_{\CC}^{-})[f](x)\|_{V_{n}^{0}}^{2} & =\sum_{i=1}^{n-1}\left|R_{\H}^{(T)}(s)\left[\left(1-\chi_{\CC}^{-}\right)\psi f_{i}\right](x)\right|^{2},\\
\|\Delta R_{\H,n}^{(T)}(s)(1-\chi_{\CC}^{-})[f](x)\|_{V_{n}^{0}}^{2} & =\sum_{i=1}^{n-1}\left|\Delta R_{\H}^{(T)}(s)\left[\left(1-\chi_{\CC}^{-}\right)\psi f_{i}\right](x)\right|^{2}.
\end{align*}
Each function $\left(1-\chi_{\CC}^{-}\right)\psi f_{i}$ is smooth
here and has has compact support depending only on $T$ and $\chi_{\CC}^{-}$.

Now using Lemma \ref{lem:sobolev-on-H} Part \ref{enu:L^2}, the fact
that $\left(1-\chi_{\CC}^{-}\right)$ is valued in $[0,1]$, using
$\psi$ is supported only on finitely many $\Gamma$-translates of
$F$, together with automorphy of $f$, we get by integrating over
$F$
\begin{align*}
\|R_{\H,n}^{(T)}(s)(1-\chi_{\CC}^{-})[f]\|_{L^{2}(F)}^{2} & \leq C\sum_{i}\|\psi f_{i}\|_{L^{2}(\H)}^{2}\leq C'\|f\|_{L^{2}(F)}^{2},\\
\|\Delta R_{\H,n}^{(T)}(s)(1-\chi_{\CC}^{-})[f]\|_{L^{2}(F)}^{2} & \leq C\sum_{i}\|\psi f_{i}\|_{L^{2}(\H)}^{2}\leq C'\|f\|_{L^{2}(F)}^{2},
\end{align*}
where $C,C'$ depend on $s,T$. Now this bound clearly extends to
$f\in L_{\phi}^{2}(\H;V_{n}^{0})$. \emph{This proves the first statement
of the lemma.}

The statement that $\LL_{\H,n}^{(T)}(s)$ is well-defined and bounded
on $L_{\phi}^{2}(\H;V_{n}^{0})$ is just an easier version of the
previous proof using Lemma \ref{lem:LH-noorm-bound} instead of Lemma
\ref{lem:sobolev-on-H}. \emph{This proves the second statement of
the lemma. }We note that we also obtain
\begin{equation}
\LL_{\H,n}^{(T)}(s)(1-\chi_{\CC}^{-})[f]=\sum_{i=1}^{n-1}\LL_{\H}^{(T)}(s)\left[\left(1-\chi_{\CC}^{-}\right)\psi f_{i}\right]e_{i}\label{eq:psi-insert-L}
\end{equation}
analogously to (\ref{eq:psi-insert}).

Now going back to (\ref{eq:psi-insert}) and using Lemma \ref{lem:sobolev-on-H}
Part \ref{enu:functional-identity} gives, \uline{considered by
restriction as equivalence classes of measurable functions on \mbox{$F$},}
\begin{align*}
 & (\Delta-s(1-s))R_{\H,n}^{(T)}(s)(1-\chi_{\CC}^{-})[f]\\
= & \sum_{i=1}^{n-1}(\Delta-s(1-s))R_{\H}^{(T)}(s)\left[\left(1-\chi_{\CC}^{-}\right)\psi f_{i}\right]e_{i}\\
= & \sum_{i=1}^{n-1}\left(\left(1-\chi_{\CC}^{-}\right)\psi f_{i}+\LL_{\H}^{(T)}(s)\left[\left(1-\chi_{\CC}^{-}\right)\psi f_{i}\right]\right)e_{i}\\
= & \left(1-\chi_{\CC}^{-}\right)f+\LL_{\H,n}^{(T)}(s)\left(1-\chi_{\CC}^{-}\right)\left[f\right].
\end{align*}
On the other hand, the fact that all functions at the two ends of
the string of equalities above satisfy the automorphy equation (\ref{eq:automorphy})
almost everywhere implies that indeed
\[
(\Delta-s(1-s))R_{\H,n}^{(T)}(s)(1-\chi_{\CC}^{-})[f]=\left(1-\chi_{\CC}^{-}\right)f+\LL_{\H,n}^{(T)}(s)\left(1-\chi_{\CC}^{-}\right)\left[f\right]
\]
as equivalence classes of measurable functions on $\H$. \emph{This
proves the final part of the lemma.}
\end{proof}
Lemma \ref{lem:parametrix-bounded} allows us to define our interior
parametrix 
\[
\M_{\phi}^{\i}(s):L_{\new}^{2}(X_{\phi})\to H_{\new}^{2}(X_{\phi})
\]
to be the operator corresponding under $L_{\new}^{2}(X_{\phi})\cong L_{\phi}^{2}(\H;V_{n}^{0})$
and $H_{\new}^{2}(X_{\phi})\cong H_{\phi}^{2}(\H;V_{n}^{0})$ to the
integral operator $R_{\H,n}^{(T)}(s)(1-\chi_{\CC}^{-}).$ The equation
(\ref{eq:automorphic-para-eq}) becomes for $s>\frac{1}{2}$
\begin{equation}
(\Delta_{X_{\phi}}-s(1-s))\M_{\phi}^{\i}(s)=(1-\chi_{\CC,\phi}^{-})+\LL_{\phi}^{\i}(s)\label{eq:interior-parametrix-equation}
\end{equation}
where 
\[
\LL_{\phi}^{\i}(s):L_{\new}^{2}(X_{\phi})\to L_{\new}^{2}(X_{\phi})
\]
is the operator corresponding under $L_{\new}^{2}(X_{\phi})\cong L_{\phi}^{2}(\H;V_{n}^{0})$
to the integral operator $\LL_{\H,n}^{(T)}(s)(1-\chi_{\CC}^{-})$
(the dependence on $T$ is suppressed here but remembered). 

Therefore if we define
\[
\M_{\phi}(s)\eqdf\M_{\phi}^{\i}(s)+\M_{\phi}^{\cusp}(s),
\]
we have 
\[
\M_{\phi}(s):L_{\new}^{2}(X_{\phi})\to H_{\new}^{2}(X_{\phi})
\]
and combining (\ref{eq:cusp-parametrix-equation}) and (\ref{eq:interior-parametrix-equation})
we obtain
\begin{align}
\left(\Delta_{X_{\phi}}-s(1-s)\right)\M_{\phi}(s) & =(1-\chi_{\CC,\phi}^{-})+\LL_{\phi}^{\i}(s)+\chi_{\CC,\phi}^{-}+\LL_{\phi}^{\cusp}(s)\nonumber \\
 & =1+\LL_{\phi}^{\i}(s)+\LL_{\phi}^{\cusp}(s).\label{eq:fundamental-identity}
\end{align}
The aim is to show that, with high probability, we can invert the
right hand side of (\ref{eq:fundamental-identity}) to define a bounded
resolvent in $s\geq s_{0}>\frac{1}{2}$. Since we have already suitably
bounded $\LL_{\phi}^{\cusp}(s)$ in Proposition \ref{prop:L-cusp-bound},
it remains to bound the operator $\LL_{\phi}^{\i}(s);$ this random
operator will be studied in detail in the next section.

\section{Random operators}

\subsection{Set up\label{subsec:Set-up}}

Suppose that $f\in C_{\phi}^{\infty}(\H;V_{n}^{0})$ with $\|f\|_{L^{2}(F)}^{2}<\infty$.
We have 
\begin{align}
\LL_{\H,n}^{(T)}(s)(1-\chi_{\CC}^{-})[f](x) & =\int_{y\in\H}\LL_{\H,n}^{(T)}(s;x,y)(1-\chi_{\CC}^{-}(y))f(y)\nonumber \\
 & =\sum_{\gamma\in\Gamma}\int_{y\in F}\LL_{\H,n}^{(T)}(s;\gamma x,y)\rho_{\phi}(\gamma^{-1})(1-\chi_{\CC}^{-}(y))f(y).\label{eq:error-inv-form}
\end{align}
Using that $\LL_{\H,n}^{(T)}(s;x,y)$ localizes to $d(x,y)\leq T+1$,
Lemma \ref{lem:geometric} implies that there is a compact set $K=K(T)\subset F$
and a finite set $S=S(T)\subset\Gamma$ such that for $x,y\in\bar{F}$
and $\gamma\in\Gamma$
\[
\LL_{\H,n}^{(T)}(s;\gamma x,y)\rho_{\phi}(\gamma^{-1})(1-\chi_{\CC}^{-}(y))=0
\]
unless $x,y\in K$ and $\gamma\in S$. 

The point of view we take in the rest of this section is that there
is an isomorphism of Hilbert spaces
\begin{align*}
L_{\phi}^{2}(\H;V_{n}^{0}) & \cong L^{2}(F)\otimes V_{n}^{0};\\
f\mapsto & \sum_{e_{i}}\langle f\lvert_{F},e_{i}\rangle_{V_{n}^{0}}\otimes e_{i}
\end{align*}
 where $e_{i}$ is an orthonormal basis of $V_{n}^{0}$ (the choice
of this basis does not matter). When conjugated by this isomorphism,
(\ref{eq:error-inv-form}) shows that
\[
\LL_{\H,n}^{(T)}(s)(1-\chi_{\CC}^{-})\cong\sum_{\gamma\in S}a_{\gamma}^{(T)}(s)\otimes\rho_{\phi}(\gamma^{-1})
\]
where 
\begin{align*}
a_{\gamma}^{(T)}(s) & :L^{2}(F)\to L^{2}(F)\\
a_{\gamma}^{(T)}(s)[f](x) & \eqdf\int_{y\in F}\LL_{\H}^{(T)}(s;\gamma x,y)(1-\chi_{\CC}^{-}(y))f(y)d\H(y).
\end{align*}
Again, this sum is supported on a finite set $S=S(T)\subset\Gamma$
depending on $T$ (and the fixed $\chi_{\CC}^{-}$).

Since $\LL_{\H}^{(T)}(s;\gamma x,y)$ is bounded depending on $T$
and smooth we have
\begin{align*}
\int_{x,y\in F}|\LL_{\H}^{(T)}(s;\gamma x,y)|^{2}(1-\chi_{\CC}^{-}(y))^{2}d\H(x)d\H(y) & =\\
\int_{x,y\in K}|\LL_{\H}^{(T)}(s;\gamma x,y)|^{2}(1-\chi_{\CC}^{-}(y))^{2}d\H(x)d\H(y) & <\infty.
\end{align*}
This shows that each integral operator $a_{\gamma}^{(T)}(s)$ is bounded,
and in fact Hilbert-Schmidt, so compact. We also produce the following
deviations estimate for the $a_{\gamma}^{(T)}(s)$.
\begin{lem}
\label{lem:deviations}For fixed $T>0$, there is a constant $c_{1}=c_{1}(T)>0$
such that for $s_{1},s_{2}\in[\frac{1}{2},1]$ and $\gamma\in S(T)$
we have 
\[
\|a_{\gamma}^{(T)}(s_{1})-a_{\gamma}^{(T)}(s_{2})\|_{L^{2}(F)}\leq c_{1}|s_{1}-s_{2}|.
\]
\end{lem}

\begin{proof}
Suppose $s_{1},s_{2}\in[\frac{1}{2},1]$, and $\gamma$ is fixed.
The operator $a_{\gamma}^{(T)}(s_{1})-a_{\gamma}^{(T)}(s_{2})$ is
a Hilbert-Schmidt operator on $L^{2}(F)$ with kernel
\[
[\LL_{\H}^{(T)}(s_{1};\gamma x,y)-\LL_{\H}^{(T)}(s_{2};\gamma x,y)](1-\chi_{\CC}^{-}(y));
\]
once again this is zero unless $x,y$ are in compact set $K(T)$. 

We have by Lemma \ref{lem:LH-bounds} Part \ref{enu:s-deriv} that
for all $x,y\in K$
\[
\left|\LL_{\H}^{(T)}(s_{1};\gamma x,y)-\LL_{\H}^{(T)}(s_{2};\gamma x,y)\right|\leq c(T)|s_{1}-s_{2}|.
\]
It follows that there is a constant $c_{1}=c_{1}(T)$ such that, writing
$\|\bullet\|_{\HS}$ for Hilbert-Schmidt norm of a Hilbert-Schmidt
operator on $L^{2}(F)$, we have
\begin{align*}
\|a_{\gamma}^{(T)}(s_{1})-a_{\gamma}^{(T)}(s_{2})\|_{L^{2}(F)} & \leq\|[\LL_{\H}^{(T)}(s_{1})-\LL_{\H}^{(T)}(s_{2})](1-\chi_{\CC}^{-})\|_{\HS}\\
 & \leq c_{1}(T)|s_{1}-s_{2}|.
\end{align*}
Finally, because $\gamma$ is in the finite set $S(T)$, the estimate
can be taken uniformly over all $\gamma\in S(T)$ by taking the maximal
value of $c_{1}$.
\end{proof}

\subsection{The operator on the free group}

We are momentarily going to apply Corollary \ref{cor:bc} to the random
operator
\[
\L_{s,\phi}^{(T)}\eqdf\sum_{\gamma\in S}a_{\gamma}^{(T)}(s)\otimes\rho_{\phi}(\gamma^{-1})
\]
Up to an intermediate step, where we approximate the coefficients
$a_{\gamma}^{(T)}(s)$ by finite rank operators, we expect to learn
that the operator norm of $\L_{s,\phi}^{(T)}$ is close to the operator
norm of 
\[
\L_{s,\infty}^{(T)}\eqdf\sum_{\gamma\in S}a_{\gamma}^{(T)}(s)\otimes\rho_{\infty}(\gamma^{-1})
\]
on $L^{2}(F)\otimes\ell^{2}(\Gamma)$, where 
\[
\rho_{\infty}:\Gamma\to\End(\ell^{2}(\Gamma)))
\]
 is the right regular representation of $\Gamma$. So we must understand
the operator $\L_{s,\infty}^{(T)}$. Under the isomorphism
\begin{align*}
L^{2}(F)\otimes\ell^{2}(\Gamma) & \cong L^{2}(\H),\\
f\otimes\delta_{\gamma} & \mapsto f\circ\gamma^{-1}
\end{align*}
(with $f\circ\gamma^{-1}$ extended by zero from a function on $\gamma F$)
the operator $\L_{s,\infty}^{(T)}$ is conjugated to none other than
\[
\LL_{\H}^{(T)}(s)(1-\chi_{\CC}^{-}):L^{2}(\H)\to L^{2}(\H)
\]
from $\S\S$\ref{subsec:Interior-parametrix-on-plane}. Since $(1-\chi_{\CC}^{-})$
is valued in $[0,1]$, multiplication by it has operator norm $\leq1$
on $L^{2}(\H)$. Hence we obtain the following corollary of Lemma
\ref{lem:LH-noorm-bound}.
\begin{cor}
\label{cor:L_infty_bound}For any $s_{0}>\frac{1}{2}$ there is $T=T(s_{0})>0$
such that for any $s\in[s_{0},1]$ we have
\[
\|\L_{s,\infty}^{(T)}\|_{L^{2}(F)\otimes\ell^{2}(\Gamma)}\leq\frac{1}{5}.
\]
\end{cor}

\subsection{Probabilistic bounds for operator norms}

The aim of this section is to prove the following result as a consequence
of Corollary \ref{cor:bc}.
\begin{prop}
\label{prop:prob-est}For any $s_{0}>\frac{1}{2}$ there is $T=T(s_{0})>0$
such that for $s$ \uline{fixed} in $[s_{0},1]$, with probability
tending to one as $n\to\infty$,
\[
\|\L_{s,\phi}^{(T)}\|_{L^{2}(F)\otimes V_{n}^{0}}\leq\frac{2}{5}.
\]
\end{prop}

\begin{proof}
Let $T$ be the one provided by Corollary \ref{cor:L_infty_bound}
for the given $s_{0}$, so that
\begin{equation}
\|\L_{s,\infty}^{(T)}\|_{L^{2}(F)\otimes\ell^{2}(\Gamma)}\leq\frac{1}{5}.\label{eq:temp-l;infty-bound}
\end{equation}
Fix $s\in[s_{0},1]$. Recall that the coefficients $a_{\gamma}^{(T)}(s)$
are supported on a finite set $S=S(T)\subset\Gamma$. Because the
$a_{\gamma}^{(T)}(s)$ are Hilbert-Schmidt, and hence compact, we
can find a \emph{finite-dimensional} subspace 
\[
W\subset L^{2}(F)
\]
 and operators $b_{\gamma}^{(T)}(s):W\to W$ such that 
\[
\|b_{\gamma}^{(T)}(s)-a_{\gamma}^{(T)}(s)\|_{L^{2}(F)}\leq\frac{1}{20|S(T)|}
\]
 for all $\gamma\in S(T)$. Therefore, since each $\rho_{\phi}(\gamma)$
is unitary on $V_{n}^{0}$ we get 
\begin{equation}
\|\L_{s,\phi}^{(T)}-\sum_{\gamma\in S}b_{\gamma}^{(T)}(s)\otimes\rho_{\phi}(\gamma^{-1})\|_{L^{2}(F)\otimes V_{n}^{0}}\leq\frac{1}{20}.\label{eq:approx-fr-1}
\end{equation}

We now apply Corollary \ref{cor:bc} to the operator $\sum_{\gamma\in S}b_{\gamma}^{(T)}(s)\otimes\rho_{\phi}(\gamma^{-1})$
to obtain that with probability $\to1$ as $n\to\infty$, 
\begin{equation}
\|\sum_{\gamma\in S}b_{\gamma}^{(T)}(s)\otimes\rho_{\phi}(\gamma^{-1})\|_{W\otimes V_{n}^{0}}\leq\|\sum_{\gamma\in S}b_{\gamma}^{(T)}(s)\otimes\rho_{\infty}(\gamma^{-1})\|_{W\otimes\ell^{2}(\Gamma)}+\frac{1}{20}.\label{eq:cd-appication}
\end{equation}
(Above, $\frac{1}{20}$ could be replaced by any $\epsilon>0$ but
this is not needed here.)

On the other hand, we also have by the same argument as led to (\ref{eq:approx-fr-1})
\begin{equation}
\|\L_{s,\infty}^{(T)}-\sum_{\gamma\in S}b_{\gamma}^{(T)}(s)\otimes\rho_{\infty}(\gamma^{-1})\|_{L(F)\otimes\ell^{2}(\Gamma)}\leq\frac{1}{20}.\label{eq:approx-fr-2}
\end{equation}
Then combining (\ref{eq:temp-l;infty-bound}), (\ref{eq:approx-fr-1}),
(\ref{eq:cd-appication}), and (\ref{eq:approx-fr-2}) gives the result.
\end{proof}

\section{Proof of Theorem \ref{thm:main-theorem}}

Given $\epsilon>0$ let $s_{0}=\frac{1}{2}+\sqrt{\epsilon}$ so that
$s_{0}(1-s_{0})=\frac{1}{4}-\epsilon$. We assume $\chi_{\CC}^{\pm}$
were chosen as in Proposition \ref{prop:L-cusp-bound} for this $s_{0}$
so that 
\begin{equation}
\|\LL_{\phi}^{\cusp}(s)\|_{L^{2}}\leq\frac{1}{5}\label{eq:Lcusp-final}
\end{equation}
for all $s\in(s_{0},1]$.

Let $T=T(s_{0})$ be the value provided by Proposition \ref{prop:prob-est}
for this $s_{0}$. To control all values of $s\in[s_{0},1]$ we use
a finite net. For $s_{1},s_{2}\in[s_{0},1]$ we have
\begin{equation}
\L_{s_{1},\phi}^{(T)}-\L_{s_{2},\phi}^{(T)}=\sum_{\gamma\in S}[a_{\gamma}^{(T)}(s_{1})-a_{\gamma}^{(T)}(s_{2})]\otimes\rho_{\phi}(\gamma^{-1})\label{eq:difff-opeartor}
\end{equation}
where $S\subset\Gamma$ is a finite set depending on $T$. Lemma \ref{lem:deviations}
tells us that for $c_{1}=c_{1}(T)>0$ we have 
\[
\|a_{\gamma}^{(T)}(s_{1})-a_{\gamma}^{(T)}(s_{2})\|_{L^{2}(F)}\leq c_{1}|s_{1}-s_{2}|
\]
for all $\gamma\in S$ and $s_{1},s_{2}\in[s_{0},1]$; hence we obtain
from (\ref{eq:difff-opeartor}) and as each $\rho_{\phi}(\gamma^{-1})$
is unitary that 
\begin{equation}
\|\L_{s_{1},\phi}^{(T)}-\L_{s_{2},\phi}^{(T)}\|\leq|S|c_{1}|s_{1}-s_{2}|.\label{eq:dev-operator}
\end{equation}
We choose a finite set $Y$, depending on $s_{0}$, of points in $[s_{0},1]$
such that every point of $[s_{0},1]$ is within 
\[
\frac{1}{5|S|c_{1}}
\]
of some element of $Y$. Now, applying Proposition \ref{prop:prob-est}
to each point of $Y$ we obtain that with probability tending to one
as $n\to\infty$ we have 
\[
\|\L_{s,\phi}^{(T)}\|_{L^{2}(F)\otimes V_{n}^{0}}\leq\frac{2}{5}
\]
\emph{for all $s\in Y$.} Hence also using (\ref{eq:dev-operator})
we obtain that with probability tending to one as $n\to\infty$ we
have 
\[
\|\L_{s,\phi}^{(T)}\|_{L^{2}(F)\otimes V_{n}^{0}}\leq\frac{3}{5}
\]
for all $s\in[s_{0},1]$.

Recall that we defined
\[
\M_{\phi}(s)\eqdf\M_{\phi}^{\i}(s)+\M_{\phi}^{\cusp}(s).
\]
By Lemma \ref{lem:parametrix-bounded} and the discussion after (\ref{eq:mcusp-def}),
for $s>\frac{1}{2}$ $\M_{\phi}(s)$ is a bounded operator from $L_{\new}^{2}(X_{\phi})$
to $H_{\new}^{2}(X_{\phi})$. We also proved in (\ref{eq:fundamental-identity})
\[
\left(\Delta_{X_{\phi}}-s(1-s)\right)\M_{\phi}(s)=1+\LL_{\phi}^{\i}(s)+\LL_{\phi}^{\cusp}(s)
\]
on $L_{\new}^{2}(X_{\phi})$. The operator $\LL_{\phi}^{\i}(s)$ is
unitarily conjugate to $\L_{s,\phi}^{(T)}$ as in $\S\S$\ref{subsec:Set-up}
and hence a.a.s. has operator norm at most $\frac{3}{5}$. Hence also
using the deterministic bound (\ref{eq:Lcusp-final}) we obtain that
\[
\|\LL_{\phi}^{\i}(s)+\LL_{\phi}^{\cusp}(s)\|_{L_{\new}^{2}(X_{\phi})}\leq\frac{4}{5}
\]
and hence a.a.s.
\[
\left(1+\LL_{\phi}^{\i}(s)+\LL_{\phi}^{\cusp}(s)\right)^{-1}
\]
 exists as a bounded operator from $L_{\new}^{2}(X_{\phi})$ to itself.
We now get 
\[
\left(\Delta_{X_{\phi}}-s(1-s)\right)\M_{\phi}(s)\left(1+\LL_{\phi}^{\i}(s)+\LL_{\phi}^{\cusp}(s)\right)^{-1}=1
\]
which shows there is a bounded right inverse to $\left(\Delta_{X_{\phi}}-s(1-s)\right)$;
this shows that $\left(\Delta_{X_{\phi}}-s(1-s)\right)$ maps $H_{\new}^{2}(X_{\phi})$
onto $L_{\new}^{2}(X_{\phi})$ and since it is self-adjoint for $s\in[\frac{1}{2},1]$,
cannot have any kernel in $L_{\new}^{2}(X_{\phi})$. This shows a.a.s.
that $\Delta_{X_{\phi}}$ has no \uline{new} eigenvalues $\lambda$
with $\lambda\leq s_{0}(1-s_{0})=\frac{1}{4}-\epsilon$. This concludes
the proof of Theorem \ref{thm:main-theorem}.

\section{Proof of Corollary \ref{cor:there-exist-compact}}

Take $X$ to be a once-punctured torus or thrice-punctured sphere
and apply Corollary \ref{cor:There-exist-finite-area} to obtain a
sequence of connected orientable finite area non-compact hyperbolic
surfaces $X_{i}$ with $\chi(X_{i})=-i$ and 
\[
\inf\left(\spec(\Delta_{X_{i}})\cap(0,\infty)\right)\to\frac{1}{4}.
\]
Suppose $X_{i}$ has genus $g_{i}$ and $n_{i}$ cusps. We have
\[
-i=\chi(X_{i})=2-2g_{i}-n_{i}
\]
which shows
\[
i\equiv\chi(X_{i})\equiv n_{i}\bmod2.
\]
In particular, for even $i$, we have an even number of cusps in $X_{i}$.
We use the following result of Buser, Burger, and Dodziuk \cite{BBD}.
See Brooks and Makover \cite[Lemma 1.1]{BrooksMakover1} for the extraction
of this lemma from \cite{BBD}.
\begin{lem}[Handle Lemma]
\label{lem:handle}Let $X$ be a finite area hyperbolic surface with
an even number of cusps $\{C_{i}\}$. It is possible to deform the
surface $X$ in a certain way to a finite area hyperbolic surface
with boundary, where each cusp becomes a bounding geodesic of length
$t$, and then glue the geodesic corresponding to $C_{2i-1}$ to the
one corresponding to $C_{2i}$ to form a family of closed hyperbolic
surfaces $X^{(t)}$ such that
\[
\limsup_{t\to0}\lambda_{1}(X^{(t)})\geq\inf\left(\spec(\Delta_{X})\cap(0,\infty)\right).
\]
\end{lem}

In particular, since each $X_{2k}$ has an even number of cusps, we
can construct a \emph{closed }hyperbolic surface $\tilde{X}_{2k}$
of Euler characteristic $-2k$ and 
\[
\lambda_{1}(\tilde{X}_{2k})\ge\inf\left(\spec(\Delta_{X_{2k}})\cap(0,\infty)\right)-\frac{1}{k}.
\]
On the other hand, the upper bound of Huber \cite{Huber} now implies
$\lambda_{1}(\tilde{X}_{2k})\to\frac{1}{4}$ as required. $\square$

\bibliographystyle{amsalpha}
\bibliography{database}

\newpage{}

\noindent Will Hide, \\
Department of Mathematical Sciences,\\
Durham University, \\
Lower Mountjoy, DH1 3LE Durham,\\
United Kingdom

\noindent \texttt{william.hide@durham.ac.uk}~\\
\texttt{}~\\

\noindent Michael Magee, \\
Department of Mathematical Sciences,\\
Durham University, \\
Lower Mountjoy, DH1 3LE Durham,\\
United Kingdom

\noindent \texttt{michael.r.magee@durham.ac.uk}\\

\end{document}